\documentclass[11pt,a4paper]{article}

\usepackage[utf8]{inputenc}
\usepackage[T1]{fontenc}
\usepackage{amsmath,amssymb,amsthm,mathtools}
\usepackage{geometry}
\usepackage{enumitem}
\usepackage{hyperref}
\hypersetup{
  hidelinks,
  pdftitle={Classification of compact Lagrangian self-similar submanifolds with Legendrian capillary boundary in the unit ball},
  pdfauthor={Dong Gao, Yong Luo, Hui Ma, Jiabin Yin},
  pdfkeywords={Lagrangian self-similar submanifold, Legendrian capillary boundary, Anciaux product, minimal Legendrian link, Calabi suspension}
}
\numberwithin{equation}{section}

\newtheorem{theorem}{Theorem}[section]
\newtheorem{lemma}[theorem]{Lemma}
\newtheorem{proposition}[theorem]{Proposition}
\newtheorem{corollary}[theorem]{Corollary}
\newtheorem{remark}[theorem]{Remark}

\newcommand{\C}{\mathbb C}
\newcommand{\R}{\mathbb R}
\newcommand{\B}{\mathbb B}
\newcommand{\Sph}{\mathbb S}
\newcommand{\Dcal}{\mathcal D}
\newcommand{\Lcal}{\mathcal L}
\newcommand{\ee}{\mathrm e}
\newcommand{\dist}{\operatorname{dist}}

\newcommand{\Id}{\operatorname{Id}}
\newcommand{\authorinfo}[3]{%
  \par\addvspace{0.6\baselineskip}%
  \noindent\textsc{#1}\par
  \noindent #2\par
  \noindent\textit{Email address:}
  \href{mailto:#3}{\texttt{#3}}\par
}

\title{Classification of compact Lagrangian self-similar submanifolds with Legendrian capillary boundary in the unit ball}
\author{Dong Gao, Yong Luo, Hui Ma, Jiabin Yin}
\date{}

\begin{document}

\maketitle

\begingroup
\renewcommand{\thefootnote}{}
\endgroup

\begin{abstract}
We classify smooth compact connected Lagrangian immersions $X$ in the closed
unit ball of $\C^n$, $n\ge2$, satisfying $H+\varepsilon X^\perp=0$,
$\varepsilon\in\{-1,0,1\}$, with Legendrian boundary on the unit sphere and
constant contact angle on each connected component.  We prove that the boundary has at
most two connected components.  When the boundary is connected, $X$ is a
diffeomorphism onto an equatorial Lagrangian $n$-disk.  When the boundary has
two components, $X$ splits globally as $X(s,p)=\gamma(s)\psi(p)$, where
$\psi$ is a compact minimal Legendrian immersion in the unit sphere and $\gamma$ is an Anciaux profile with a
unique radial minimum. The two contact angles are supplementary.  
In complex dimension two, every non-disk solution is a finite cover of a
Lagrangian catenoid segment for $\varepsilon=0$ or of a rotational Anciaux
annulus for $\varepsilon=\pm1$.  In higher complex dimensions, iterated Calabi
suspensions produce families whose minimal Legendrian links have nontrivial
topology.
\end{abstract}

\section{Introduction}

We study a global classification problem for compact Lagrangian
submanifolds with Legendrian capillary boundary in the unit ball.  Let
$
        X:M^n\longrightarrow\C^n
$
be a smooth Lagrangian immersion of a compact connected manifold with nonempty
boundary.  We consider Lagrangian self-similar submanifolds satisfying 
\begin{equation}\label{eq:intro-self-similar}
        H+\varepsilon X^\perp=0,
        \qquad \varepsilon\in\{-1,0,1\}.
\end{equation}
Here $H$ denotes the trace of the second fundamental form. Self-similar submanifolds arise naturally in mean curvature flow and play an important role in the analysis of singularities (see \cite{Huisken1990}). The values
$\varepsilon=0,1,-1$ correspond respectively to the minimal, shrinking, and
expanding cases.  Other nonzero coefficients are obtained by homothety, with a
corresponding change in the radius of the supporting sphere. 

When the ambient manifold is K\"ahler-Einstein, Smoczyk \cite{Smoczyk1996} proved that the Lagrangian condition is preserved by the mean curvature flow; the resulting evolution is known as Lagrangian mean curvature flow. Lagrangian self-shrinkers and self-expanders have since been studied extensively, both from the viewpoint of classification and through explicit constructions; see \cite{Anciaux2006, ArezzoSun2013, CastroLerma2014, ChauChenYuan2012, ChengHiroakiWei2022, Dingxin2014, JoyceLeeTsui2010, LeeWang2009,LeeWang2010,  LiWang2017, 
LotayNeves2013} and the references therein.

Motivated by the theory of capillary hypersurfaces, and in particular by the extensive study of free-boundary minimal hypersurfaces (see \cite{Li2020} for a survey), Li--Wang--Weng \cite{LiWangWeng2021} introduced the Legendrian capillary boundary condition for Lagrangian submanifolds in the unit ball of $\C^n$. Assume that $X(\partial M)\subset\Sph^{2n-1}$, and let $\nu$ be the outward
unit conormal of $\partial M$ in $M$.  
The boundary is Legendrian precisely when
\begin{equation}\label{eq:intro-capillary}
        \nu=\sin\theta\,X+\cos\theta\,JX
        \qquad\text{on }\partial M
\end{equation}
for an angle function $\theta:\partial M\to\mathbb R/2\pi\mathbb Z$.  The
boundary is called Legendrian capillary if $\theta$ is constant on each connected
component.  The free-boundary condition corresponds to $\theta=\pi/2$,
equivalently $\nu=X$. 

For minimal Lagrangian surfaces with Legendrian capillary boundary in $\overline{\B}^4$, Li--Wang--Weng
\cite{LiWangWeng2021} proved a Nitsche-type disk rigidity theorem.  Luo--Sun
\cite{LuoSun2021} established a Legendrian free-boundary rigidity theorem and classified
annulus-type minimal Lagrangian surfaces with Legendrian capillary boundary
as Lagrangian catenoids. 

More recently, Gao--Ma--Yao \cite{GaoMaYao2026} established equatorial-disk rigidity for Lagrangian  self-similar submanifolds in arbitrary dimension under the vanishing of the relative Liouville class and under several natural boundary hypotheses, including the free-boundary condition and connected Legendrian capillary boundary.  They also constructed compact non-disk examples with two Legendrian capillary boundary components.

Within the capillary setting considered here, these results give the
rigidity conclusion for connected boundary and show that the two-boundary
case is nonempty.  Two global classification questions remained: can the
boundary have more than two connected components, and must every
two-boundary solution be of Anciaux product form?  We answer both questions by
proving that the boundary has at most two connected components and that every
two-boundary solution is globally an Anciaux product determined by a planar
profile curve and a compact minimal Legendrian link in the unit sphere.

We first recall the Anciaux products appearing in the classification  \cite{Anciaux2006}. Let
$
        \psi:N^{n-1}\rightarrow\Sph^{2n-1}
$
be a compact minimal Legendrian immersion, and let
$\gamma:I\to\C^*:=\C\setminus\{0\}$ be a smooth regular planar curve.  Then
\begin{equation}\label{eq:intro-anciaux-product}
        F(s,p)=\gamma(s)\psi(p)
\end{equation}
defines a Lagrangian immersion.  For these product immersions, equation
\eqref{eq:intro-self-similar} reduces to an ordinary differential system
for $\gamma$.  When $\varepsilon=0$ and the link is a round Legendrian
sphere, this construction recovers the Lagrangian catenoid family
\cite{Anciaux2006,CastroUrbano1999,GaoMaYao2026,HarveyLawson1982}.  More generally, special Lagrangian cone theory provides compact minimal Legendrian links with nontrivial topology \cite{Haskins2004, HaskinsKapouleas2007}. Using such links in \eqref{eq:intro-anciaux-product} we obtain higher-dimensional capillary self-similar examples with nonspherical boundary topology.  

For the main results, fix $n\ge2$ and
$\varepsilon\in\{-1,0,1\}$, and let
$X:M^n\to\C^n$ be a smooth Lagrangian immersion of a compact connected
manifold with nonempty boundary satisfying
\eqref{eq:intro-self-similar}.  Assume that
\begin{equation}\label{eq:intro-ball}
        X(M)\subset\overline{\B}^{2n},
        \qquad
        X(\partial M)\subset\Sph^{2n-1},
\end{equation}
and that each boundary component is Legendrian capillary.

These hypotheses also imply the strict interior inclusion and fix the
contact-angle normalization.  Indeed, for $\rho=|X|^2/2$,
\[
        \Delta\rho
        =n+\langle X,H\rangle
        =n-\varepsilon|X^\perp|^2
        \ge n-1>0,
\]
where the last inequality uses $|X^\perp|\le |X|\le1$.  Since
$\rho\le1/2$ on $M$ and $\rho=1/2$ on $\partial M$, the strong maximum
principle and the Hopf boundary point lemma give
\[
        X(M^\circ)\subset\B^{2n},
        \qquad
        \sin\theta
        =\langle X,\nu\rangle
        =\partial_\nu\rho>0
        \quad\text{on }\partial M.
\]
Thus each contact angle has a unique representative in $(0,\pi)$, which
we use throughout. 

\begin{theorem}\label{thm:boundary-components}
Let
$X:M^n\to\C^n$ be a smooth Lagrangian immersion of a compact connected
manifold with nonempty boundary satisfying
\eqref{eq:intro-self-similar} and \eqref{eq:intro-ball}.
Assume that each boundary component is Legendrian capillary. Then \(\partial M\) has at most two connected components.
\end{theorem}

\begin{theorem}\label{thm:classification}
Under the assumptions of Theorem~\ref{thm:boundary-components}, exactly one of the following occurs.

\begin{enumerate}[label=\rm(\roman*)]
\item The boundary is connected, $M$ is diffeomorphic to the closed
$n$-ball, and $X$ is a diffeomorphism onto an equatorial Lagrangian $n$-disk.

\item The boundary has two connected components.  There exist a closed
connected $(n-1)$-manifold $N$, a compact minimal Legendrian immersion
$\psi:N\to\Sph^{2n-1}$, an arclength-parametrized curve
$\gamma=r\ee^{i\phi}:[0,L]\to\C^*$ with $L>0$, and a diffeomorphism
$\Phi:[0,L]\times N\to M$ such that
\begin{equation}
        X\circ\Phi(s,p)=\gamma(s)\psi(p).
\end{equation}

Moreover, writing $\gamma'=\ee^{i\alpha}$ and
$\Theta=\alpha-\phi$, the profile curve satisfies
\begin{equation}\label{eq:intro-anciaux-system}
        r'=\cos\Theta,
        \qquad
        \Theta'=\left(\varepsilon r-\frac nr\right)\sin\Theta,
\end{equation}
with the nonzero first integral
\begin{equation}
        r^n\ee^{-\varepsilon r^2/2}\sin\Theta=E,
        \qquad E\in\mathbb R\setminus\{0\}.
\end{equation}
The radial function satisfies
\begin{equation}
        r(0)=r(L)=1,
        \qquad
        0<r(s)<1\quad(0<s<L),
\end{equation}
and attains its minimum $r_*\in(0,1)$ at a unique point of $(0,L)$.
The two boundary contact angles are supplementary:
\begin{equation}
        \theta_0+\theta_1=\pi.
\end{equation}
\end{enumerate}
\end{theorem}

The first integral also determines $|\cos\theta_i|$ from the unique
minimum radius; the precise relation is recorded in
Corollary~\ref{cor:angle-formula}.

In dimension two, $N\cong S^1$, so the non-disk case is an annulus.

\begin{corollary}
\label{cor:two-dimensional}
Let $n=2$.  Every non-disk solution is, after a unitary transformation, of
the form
\begin{equation}\label{eq:intro-two-dimensional}
        X(s,t)=\gamma(s)(\cos mt,\sin mt),
        \qquad
        (s,t)\in[0,L]\times\R/(2\pi\mathbb Z),
\end{equation}
for an integer $m\ge1$.  For $\varepsilon=0$ these are finite covers of
Lagrangian catenoid segments; for $\varepsilon=\pm1$ they are finite covers of
the corresponding rotational Anciaux annuli.
\end{corollary}

The rigidity of the Legendrian link in
Corollary~\ref{cor:two-dimensional} is specific to dimension two: it is a
possibly multiply covered great circle.  In higher dimensions, compact
minimal Legendrian links may have nontrivial topology.  The next result
shows that Calabi suspension propagates such topology and that the resulting
links occur in compact capillary self-similar solutions.

\begin{theorem}
\label{thm:intro-calabi-realization}
Let $k_0\ge2$, and let
$\psi_{k_0}:N^{k_0-1}\to\Sph^{2k_0-1}$ be a compact connected minimal
Legendrian immersion.  For every $n\ge k_0$, iterated Calabi suspension
yields a compact minimal Legendrian immersion
\[
        \psi_n:\mathbb T^{n-k_0}\times N^{k_0-1}
        \longrightarrow\Sph^{2n-1}.
\]
For each $\varepsilon\in\{-1,0,1\}$, combining $\psi_n$ with the Anciaux
profiles produces a one-parameter family of compact immersed Lagrangian
solutions of $H+\varepsilon X^\perp=0$ in $\overline{\B}^{2n}$.  Each member
is defined on
$[0,L]\times\mathbb T^{n-k_0}\times N^{k_0-1}$ for some $L>0$, its interior
lies in $\B^{2n}$, and it has exactly two Legendrian capillary boundary
components, both diffeomorphic to
$\mathbb T^{n-k_0}\times N^{k_0-1}$, with supplementary contact angles.
\end{theorem}

The proof separates the free and non-free branches.  The free-boundary
case follows from \cite[Theorem~1.3]{GaoMaYao2026}, and it remains to
consider a non-free boundary component $\Sigma_0$.  Boundary identities
show that the boundary immersion $X|_{\Sigma_0}$ is minimal Legendrian.
The boundary immersion and its contact angle determine an Anciaux model
with the same Cauchy data, and one-sided Cauchy uniqueness identifies $X$
with this model on a collar.  There, $d\rho$ and the pulled-back Liouville
form $\lambda$ are collinear, and real analyticity extends this rank-one
identity throughout $M^\circ$.  The identity yields an intrinsic splitting
and the profile system without any symmetry or product ansatz.  The radial
dynamics has a unique turning point and returns to the unit sphere in finite
time.  Flowing $\Sigma_0$ along the intrinsic line field then exhausts $M$
and yields the global product, the boundary count, and the classification.

\vspace{0.2cm}

The paper is organized as follows.
Section~\ref{sec:preliminaries} develops the boundary geometry,
analyticity, and one-sided Cauchy uniqueness.
Section~\ref{sec:rank-one} constructs the Anciaux collar and derives the
global rank-one identity and intrinsic splitting.
Section~\ref{sec:global-classification} analyzes the radial dynamics and
intrinsic flow to prove Theorems~\ref{thm:boundary-components}
and~\ref{thm:classification}.
Section~\ref{sec:calabi-realizations} establishes the two-dimensional
classification and the angle estimate before constructing the Calabi-suspended families.

\section{Preliminaries and boundary geometry}\label{sec:preliminaries}

Throughout the paper, $D$ denotes the Euclidean connection of $\C^n$, $\nabla$ the Levi-Civita connection of the induced metric on $M$, and $B$ the second fundamental form.  Thus
$$
        D_UV=\nabla_UV+B(U,V)
$$
for tangent vector fields $U,V$ on $M$.  The mean-curvature vector is
$$
        H=\sum_{j=1}^n B(e_j,e_j)
$$
for any local orthonormal tangent frame $\{e_j\}_{j=1}^n$.

\subsection{Basic Lagrangian identities}

We use the standard K\"ahler form
$$
        \omega(U,V)=\langle JU,V\rangle
$$
and the Liouville form
\begin{equation}\label{eq:ambient-liouville}
        \lambda_{\C^n}
        =\sum_{j=1}^n(x_j\,dy_j-y_j\,dx_j),
        \qquad
        d\lambda_{\C^n}=2\omega.
\end{equation}
For a Lagrangian immersion $X:M\to\C^n$, we write
\begin{equation}
        \lambda=X^*\lambda_{\C^n}.
\end{equation}
With the normalization \eqref{eq:ambient-liouville},
\begin{equation}\label{eq:lambda-evaluation}
        \lambda(Y)=\langle JX,Y\rangle
\end{equation}
for every tangent vector $Y$.  Since $X^*\omega=0$,
\begin{equation*}
        d\lambda=0.
\end{equation*}
If $X(\partial M)\subset\Sph^{2n-1}$ is Legendrian, then
\begin{equation}\label{eq:lambda-boundary-zero}
        \lambda|_{T\partial M}=0.
\end{equation}
Indeed, the standard contact form on the unit sphere is the restriction of $\lambda_{\C^n}$, up to the harmless normalization factor used in different conventions.

Since $\lambda|_{T\partial M}=0$, the Liouville form defines a relative
de Rham class \cite{GaoMaYao2026}
\[
        \mathcal L_{\mathrm{rel}}(X)
        :=[\lambda]_{\mathrm{rel}}
        \in H^1_{\mathrm{dR}}(M,\partial M;\mathbb R).
\]

We recall the standard total symmetry of the Lagrangian cubic form.

\begin{lemma}
Let $X:M^n\to\C^n$ be a Lagrangian immersion.  Define
\begin{equation}
        C(U,V,W)=\langle B(U,V),JW\rangle.
\end{equation}
Then $C$ is symmetric in all three variables.
\end{lemma}

\begin{proof}
The symmetry in the first two variables follows from the symmetry of $B$.
Differentiating $\langle JU,V\rangle=0$ in the direction of $W$ and using
$DJ=0$ and the Gauss formula gives
\[
        \langle B(W,U),JV\rangle
        =
        \langle B(W,V),JU\rangle.
\]
Thus $C(W,U,V)=C(W,V,U)$.  Together with the symmetry in the first two
variables, this proves that $C$ is totally symmetric.
\end{proof}

The following identities relate the self-similar equation to the Liouville form.  

\begin{lemma}[\cite{GaoMaYao2026}]
Let $X:M^n\to\C^n$ be a Lagrangian immersion satisfying \eqref{eq:intro-self-similar}.  Define
$$
        \alpha_H(Y)=\langle JH,Y\rangle,
        \qquad
        f_\varepsilon=\frac{\varepsilon|X|^2}{2}.
$$
Then
\begin{equation}\label{eq:alphaH-lambda}
        \alpha_H=-\varepsilon\lambda
\end{equation}
and
\begin{equation}
        \delta_{f_\varepsilon}\lambda=0,
        \qquad
        where \ \ \delta_{f_\varepsilon}\eta:=\delta\eta+\eta(\nabla f_\varepsilon).
\end{equation}
\end{lemma}

Let
$$
        \Omega=dz_1\wedge\cdots\wedge dz_n
$$
be the standard holomorphic volume form.  On an oriented Lagrangian immersion, the Lagrangian angle $\beta_M$ is locally defined by
\begin{equation}
        \Omega(e_1,\ldots,e_n)=\ee^{i\beta_M}
\end{equation}
for an oriented orthonormal tangent frame. With our convention for $H$,
\begin{equation}\label{eq:phase-mean-curvature}
        \alpha_H=-d\beta_M.
\end{equation}
Indeed, the standard formula is $H=J\nabla\beta_M$, and therefore
$$
        \alpha_H(Y)
        =\langle J(J\nabla\beta_M),Y\rangle
        =-d\beta_M(Y).
$$
Combining \eqref{eq:alphaH-lambda} and \eqref{eq:phase-mean-curvature}, one obtains on every simply connected coordinate neighborhood
\begin{equation}\label{eq:d-beta-lambda}
        d\beta_M=\varepsilon\lambda.
\end{equation}
This local identity is all that is needed below, and no vanishing Maslov class is assumed.

Let $\Sigma^{n-1}\to\Sph^{2n-1}$ be an oriented Legendrian immersion.  If $e_1,\ldots,e_{n-1}$ is an oriented orthonormal tangent frame, define the Legendrian angle $\beta_\Sigma$ locally by
\begin{equation*}
        \Omega(X,e_1,\ldots,e_{n-1})=\ee^{i\beta_\Sigma}.
\end{equation*}
The following identity relates the Lagrangian and Legendrian angles along the boundary.

\begin{lemma}\label{lem:phase-relation}
Let $X:M^n\to\overline{\B}^{2n}$ be an oriented Lagrangian immersion, and let $\Sigma\subset\partial M$ be an oriented Legendrian boundary component.  Choose the orientation of $\Sigma$ so that $(e_1,\ldots,e_{n-1},\nu)$ is an oriented frame of $M$.  If $\nu$ satisfies \eqref{eq:intro-capillary}, then
\begin{equation}\label{eq:phase-relation}
        \beta_M
        =\beta_\Sigma+\frac\pi2-\theta+(n-1)\pi
        \quad\bmod 2\pi.
\end{equation}
Consequently,
\begin{equation}\label{eq:d-phase-relation}
        d\beta_M|_{T\Sigma}
        =d\beta_\Sigma-d\theta.
\end{equation}
\end{lemma}

\begin{proof}
By definition and the capillary decomposition,
$$
\begin{aligned}
        \ee^{i\beta_M}
        &=\Omega(e_1,\ldots,e_{n-1},\nu)\\
        &=\sin\theta\,\Omega(e_1,\ldots,e_{n-1},X)
          +\cos\theta\,\Omega(e_1,\ldots,e_{n-1},JX).
\end{aligned}
$$
The complex linearity of $\Omega$ in its last argument gives
$$
        \Omega(e_1,\ldots,e_{n-1},JX)
        =i\,\Omega(e_1,\ldots,e_{n-1},X).
$$
Moving $X$ from the last position to the first contributes $(-1)^{n-1}$, hence
$$
        \Omega(e_1,\ldots,e_{n-1},X)
        =(-1)^{n-1}\ee^{i\beta_\Sigma}.
$$
It follows that
$$
\begin{aligned}
        \ee^{i\beta_M}
        &=(-1)^{n-1}(\sin\theta+i\cos\theta)\ee^{i\beta_\Sigma}\\
        &=\ee^{i((n-1)\pi+\pi/2-\theta)}\ee^{i\beta_\Sigma}.
\end{aligned}
$$
This proves \eqref{eq:phase-relation}.  Differentiating a local real-valued lift gives \eqref{eq:d-phase-relation}.
\end{proof}

The following equivalence is contained in
\cite[Section~2.2]{Haskins2004}.

\begin{lemma}\label{lem:legendrian-minimal-phase}
Let
$$
        \psi:N^{n-1}\longrightarrow\Sph^{2n-1}
$$
be an oriented connected Legendrian immersion.  The following are equivalent:
\begin{enumerate}[label=\rm(\roman*)]
\item the Legendrian angle is locally constant;
\item the cone
$$
        C(N)=\{r\psi(p):r>0,\ p\in N\}
$$
is special Lagrangian, with some constant phase;
\item $\psi$ is minimal in $\Sph^{2n-1}$.
\end{enumerate}
\end{lemma}

\begin{proof}
For completeness, we give the short proof.  Consider the cone immersion
\begin{equation*}
\begin{aligned}
        F:(0,\infty)\times N&\longrightarrow\C^n,\\
        (r,p)&\longmapsto r\psi(p).
\end{aligned}
\end{equation*}
If $e_1,\ldots,e_{n-1}$ is an oriented orthonormal frame on $N$, then
an oriented orthonormal frame for the tangent space of the cone is
$$
        E_0=\psi,
        \qquad
        E_a=\psi_*e_a,
        \qquad
        1\leq a\leq n-1.
$$
The Legendrian condition implies that the cone is Lagrangian.  

Moreover,
by the definition of the Legendrian angle,
$$
\begin{aligned}
        &\Omega(E_0,E_1,\ldots,E_{n-1})
        \\&=
        \Omega(\psi,\psi_*e_1,\ldots,\psi_*e_{n-1})\\
        &=
        \ee^{i\beta_N}.
\end{aligned}
$$
Hence the Lagrangian angle of the cone agrees with the Legendrian angle
of the link, modulo $2\pi$.  It follows that the Legendrian angle is
locally constant if and only if the cone is special Lagrangian, proving
the equivalence of \rm(i) and \rm(ii).

For a Lagrangian immersion in $\C^n$, the Lagrangian angle is locally
constant if and only if the mean-curvature vector vanishes, by
\eqref{eq:phase-mean-curvature}. Thus \rm(ii) is equivalent to the
minimality of the cone.

Finally, the Euclidean mean-curvature vector of a cone is $r^{-1}$ times
the spherical mean-curvature vector of its link, viewed in the common
normal bundle.  
Consequently, the cone is minimal in $\C^n$ if and only
if its link is minimal in $\Sph^{2n-1}$.  This proves the equivalence of \rm(ii) and \rm(iii).
\end{proof}

\subsection{Analyticity and one-sided Cauchy uniqueness}

The next result is a direct application of Morrey's analytic regularity theorem for analytic quasilinear elliptic systems \cite[Chapter~6]{Morrey1966}.  We give the graph equation explicitly because the same form is needed in the boundary uniqueness argument.

\begin{lemma}\label{lem:analyticity}
Fix $\varepsilon\in\{-1,0,1\}$. Let
$$
        X:M^n\longrightarrow\mathbb R^N
$$
be a smooth immersed solution of
$$
        H+\varepsilon X^\perp=0.
$$
Then $M^\circ$ admits a real-analytic atlas compatible with its smooth
structure such that $X|_{M^\circ}$ is real analytic.  Equivalently, near
every interior point, the corresponding local sheet of $X(M^\circ)$ is
a real-analytic graph in $\mathbb R^N$.
\end{lemma}

\begin{proof}
Fix $q\in M^\circ$, and set
$$
        p=X(q),
        \qquad
        P=dX_q(T_qM).
$$
Since $X$ is an immersion, it is an embedding on a sufficiently small
neighborhood $U$ of $q$.  Moreover, the differential at $q$ of
$$
        \pi_P\circ(X-p):U\longrightarrow P
$$
is an isomorphism, where \(\pi_P:\mathbb R^N\to P\) denotes the
orthogonal projection.  After decreasing $U$, the inverse function theorem
therefore gives a coordinate domain $V\subset P$ in which the image is
the graph
$$
        Y(z)=p+z+u(z),
        \qquad
        u:V\longrightarrow P^\perp.
$$

Choose orthonormal bases $\{e_i\}_{i=1}^n$ of $P$ and
$\{E_\alpha\}_{\alpha=1}^{N-n}$ of $P^\perp$, and write
$$
        u=u^\alpha E_\alpha.
$$
Then
$$
        Y_i=e_i+u_i^\alpha E_\alpha,
        \qquad
        g_{ij}
        =\delta_{ij}+\sum_\alpha u_i^\alpha u_j^\alpha.
$$
A convenient normal frame is
$$
        N_\alpha
        =E_\alpha-u_i^\alpha e_i,
        \qquad
        \alpha=1,\ldots,N-n.
$$
Taking the inner product of
$$
        H+\varepsilon Y^\perp=0
$$
with $N_\alpha$ gives
\begin{equation}\label{eq:graph-self-similar-system}
        g^{ij}(Du)u_{ij}^\alpha
        +\varepsilon
        \left\langle p+z+u,N_\alpha(Du)\right\rangle
        =0,
        \qquad
        \alpha=1,\ldots,N-n.
\end{equation}
Indeed,
$$
        \langle H,N_\alpha\rangle
        =g^{ij}(Du)u_{ij}^\alpha,
$$
while
$$
        \langle Y^\perp,N_\alpha\rangle
        =\langle Y,N_\alpha\rangle.
$$

The principal symbol of \eqref{eq:graph-self-similar-system} is
$$
        g^{ij}(Du)\xi_i\xi_j\,
        \operatorname{Id}_{P^\perp}.
$$
After decreasing $V$ if necessary, the system is uniformly elliptic.
Its coefficients are real analytic in $(z,u,Du)$.  Morrey's analytic
regularity theorem for analytic quasilinear elliptic systems therefore
implies that $u$ is real analytic.

It remains to check that the graphical coordinates obtained at
different points are analytically compatible.  For $q'\in M^\circ$,
write
$$
        P_{q'}=dX_{q'}(T_{q'}M),
        \qquad
        y_{q'}=\pi_{P_{q'}}\circ\bigl(X-X(q')\bigr).
$$
On the overlap of two graphical coordinate neighborhoods,
$$
\begin{aligned}
        y_{q'}\circ y_q^{-1}(z)
        &=
        \pi_{P_{q'}}
        \left(
        X\circ y_q^{-1}(z)-X(q')
        \right).
\end{aligned}
$$
The map $X\circ y_q^{-1}$ is real analytic by the first part of the
proof, and $\pi_{P_{q'}}$ is linear.  Hence every transition map is
real analytic.  The graphical coordinates therefore define a
real-analytic atlas compatible with the original smooth structure, and
$X$ is real analytic with respect to this atlas.
\end{proof}

We shall use the following vector-valued unique-continuation principle.  It is the form obtained from the Carleman estimate in Aronszajn \cite{Aronszajn1957}; the same argument applies componentwise to systems with a common scalar principal operator. An exterior-form version was developed by Aronszajn--Krzywicki--Szarski \cite{Aronszajn1962}.

\begin{lemma}\label{lem:ucp-external}
Let $w\in W^{2,2}_{\mathrm{loc}}(\Omega;\R^m)$, where $\Omega\subset\R^n$ is connected.  Suppose
\begin{equation}\label{eq:ucp-inequality}
        |A^{ij}\partial_{ij}w|
        \le C\bigl(|w|+|Dw|\bigr)
\end{equation}
almost everywhere, where $A^{ij}=A^{ji}$ are Lipschitz and uniformly elliptic and the same scalar operator $A^{ij}\partial_{ij}$ acts on each component.  If $w$ vanishes on a nonempty open subset of $\Omega$, then $w\equiv0$ in $\Omega$.
\end{lemma}

\begin{proof}
Aronszajn's Carleman estimate applies to the scalar operator
$A^{ij}\partial_{ij}$ with Lipschitz leading coefficients
\cite{Aronszajn1957}.  Apply that estimate to each component $w^\alpha$ and
sum over $\alpha$.  Because the principal operator is the same on every
component, the left-hand sides add without producing mixed second-order
terms, while the right-hand side is bounded by
$C(|w|+|Dw|)$.  The usual absorption argument therefore gives strong unique
continuation for the vector $w$.  In particular, vanishing on a nonempty open
set implies vanishing to infinite order at every boundary point of that set,
and strong unique continuation propagates the zero set through the connected
domain.  Hence $w\equiv0$ on $\Omega$.  The same componentwise summation is
also implicit in the exterior-form version of
Aronszajn--Krzywicki--Szarski~\cite{Aronszajn1962}.
\end{proof}
\begin{remark}
Lemma~\ref{lem:ucp-external} is the only unique-continuation input used in the paper.  In the application below the coefficients are smooth, and the common scalar principal part follows directly from the graph system \eqref{eq:graph-self-similar-system}.
\end{remark}

\begin{lemma}\label{lem:cauchy}
Fix $\varepsilon\in\{-1,0,1\}$.  For $a=1,2$, let
$$
        X_a:M_a^n\longrightarrow\mathbb R^N
$$
be smooth immersed solutions of
$$
        H_a+\varepsilon X_a^\perp=0.
$$
Let $\Sigma_a$ be compact boundary components of $M_a$, let
$$
        F:\Sigma_1\longrightarrow\Sigma_2
$$
be a diffeomorphism, and let $\eta_a$ denote the inward unit conormal
along $\Sigma_a$.  Suppose that, for every $p\in\Sigma_1$,
$$
        X_1(p)=X_2(F(p))
$$
and
$$
        dX_1|_p\bigl(\eta_1(p)\bigr)
        =
        dX_2|_{F(p)}\bigl(\eta_2(F(p))\bigr).
$$
Then there exist one-sided neighborhoods $V_a$ of $\Sigma_a$ and a
diffeomorphism
$$
        \Phi:V_1\longrightarrow V_2
$$
such that
$$
        \Phi|_{\Sigma_1}=F,
        \qquad
        X_1=X_2\circ\Phi
        \quad\text{on }V_1.
$$

\end{lemma}

\begin{proof}
We work sheetwise near a fixed point $p_1\in\Sigma_1$ and its corresponding
point $p_2=F(p_1)\in\Sigma_2$.  The possible presence of other preimages of
the same ambient point is irrelevant, because an immersion is an embedding on
a sufficiently small neighborhood of each prescribed preimage.  Shrink the
two neighborhoods so that both chosen sheets are embedded.

Let $P$ be their common tangent $n$-plane at the boundary point, and let
$\pi_P$ be orthogonal projection onto $P$.  The differential of
$\pi_P\circ X_a$ is an isomorphism at $p_a$; hence, after shrinking again,
each sheet is a graph over a one-sided domain in $P$.  The equality of the
boundary parametrizations implies that the projected boundary hypersurfaces
are the same.  Equality of the specified inward conormals implies that the two
projected sheets lie on the same side of that hypersurface.  We may therefore
choose one connected one-sided neighborhood $\Omega^+\subset P$, contained in
both projected domains, and write
\begin{equation}
        Y_a(y)=y+u_a(y),
        \qquad
        u_a:\Omega^+\longrightarrow P^\perp,
        \qquad a=1,2.
\end{equation}
Here the origin and a fixed translation have been absorbed into the affine
coordinate $y$.

Let $\Gamma=\partial\Omega^+$ denote the common projected boundary portion.
On $\Gamma$, equality of the boundary immersions gives $u_1=u_2$.  For a graph
over a fixed plane, its tangent space is the graph of the unique linear map
$Du_a:P\to P^\perp$.  Since the full tangent $n$-planes of the two immersed
sheets agree at every corresponding boundary point, it follows that
$Du_1=Du_2$ on $\Gamma$.  The conormal hypothesis is used here to select the
same one-sided graph; it rules out comparing the upper sheet of one immersion
with the lower sheet of the other.  Thus, with $w=u_1-u_2$,
\begin{equation}
        w=0,
        \qquad
        Dw=0
        \qquad\ \text{on }\Gamma.
\end{equation}

Choose a smooth diffeomorphism
$\chi:B_r^+\to\Omega^+$ which maps $I_r=\{y_n=0\}$ onto $\Gamma$, and replace
$u_a$ by $u_a\circ\chi$.  Because the same fixed change of independent
variables is used for both systems, the transformed principal operator remains
a scalar uniformly elliptic operator acting identically on every normal
component.  More explicitly, the affine graph equation
\begin{equation*}
        g^{ij}(Du_a)\partial_{ij}u_a^\alpha
        +\varepsilon\langle y+u_a,N_\alpha(Du_a)\rangle=0
\end{equation*}
transforms into a quasilinear system whose leading coefficient is
\begin{equation*}
        \widehat g_a^{kl}
        =g_a^{ij}\,\partial_i(\chi^{-1})^k\,\partial_j(\chi^{-1})^l,
\end{equation*}
while the second derivatives of $\chi$ occur only in lower-order terms.
After this common flattening, we have
\begin{equation}\label{eq:full-cauchy-data}
        w=0,
        \qquad
        Dw=0
        \qquad\ \text{on }I_r.
\end{equation}

Subtract the two transformed quasilinear systems.  For the leading terms,
use the fundamental theorem of calculus along the segment
$u_2+t(u_1-u_2)$ in the variables $(u,Du)$.  The result is a linear system
\begin{equation}
        A^{ij}(y)\partial_{ij}w^\alpha
        +B^i_{\alpha\beta}(y)\partial_iw^\beta
        +C_{\alpha\beta}(y)w^\beta=0,
\end{equation}
where $A^{ij}=A^{ji}$ is smooth, uniformly elliptic, and independent of the
normal index $\alpha$.  Consequently,
\begin{equation}
        |A^{ij}\partial_{ij}w|
        \le C\bigl(|w|+|Dw|\bigr)
        \qquad\text{in }B_r^+.
\end{equation}

Define the zero extension
\begin{equation}
        \widetilde w(y)=
        \begin{cases}
        w(y),&y_n\ge0,\\
        0,&y_n<0.
        \end{cases}
\end{equation}
We verify the required Sobolev regularity.  For a scalar component and a test
function $\zeta$, integration by parts in the upper half-ball gives
\begin{equation}
        \langle\partial_i\widetilde w,\zeta\rangle
        =\int_{B_r^+}(\partial_iw)\zeta,
\end{equation}
because the boundary term contains the trace of $w$, which is zero.  A second
integration by parts gives
\begin{equation}\label{eq:second-extension-derivative}
        \langle\partial_{ij}\widetilde w,\zeta\rangle
        =\int_{B_r^+}(\partial_{ij}w)\zeta,
\end{equation}
because the remaining interface term contains the trace of $\partial_jw$,
which also vanishes.  Hence $\widetilde w\in W^{2,2}_{\mathrm{loc}}(B_r)$ and
its weak second derivatives are precisely the zero extensions of those of
$w$.

Extend $A^{ij}$ to Lipschitz uniformly elliptic coefficients on $B_r$, and
extend the lower-order coefficients boundedly.  Equations
\eqref{eq:full-cauchy-data}--\eqref{eq:second-extension-derivative} show that
no distribution supported on $I_r$ occurs, so $\widetilde w$ satisfies an
inequality of the form \eqref{eq:ucp-inequality} in the full ball.  Since
$\widetilde w$ vanishes on the open lower half-ball,
Lemma~\ref{lem:ucp-external} yields $\widetilde w\equiv0$ in a smaller ball.
Thus the two prescribed sheets agree near the chosen boundary point.

It remains to patch the local identifications.  Cover the compact boundary
component by finitely many sheetwise embedding neighborhoods.  On an overlap,
the two local identifications both fix the boundary parametrization and map
one chosen embedded sheet onto the same chosen embedded sheet.  After
shrinking the cover, local injectivity of either immersion forces the two
identifications to coincide on the overlap.  They therefore patch to a
one-sided collar diffeomorphism.  Finally compactness permits one common positive
collar width, which completes the proof.
\end{proof}

\subsection{Legendrian capillary boundary geometry}

Let $\Sigma$ be a connected component of $\partial M$.  Throughout this section, $\nu$ denotes the outward unit conormal, and $U,V,W$ denote tangent vectors to $\Sigma$.  The contact angle $\theta$ is constant on $\Sigma$.

We will use the following identities on the boundary.
\begin{lemma}
Along $\Sigma$, for every $U\in T\Sigma$,
\begin{equation}\label{eq:boundary-nabla-nu}
        \nabla_U\nu=\sin\theta\,U
\end{equation}
and
\begin{equation}\label{eq:boundary-B-Unu}
        B(U,\nu)=\cos\theta\,JU.
\end{equation}
Moreover,
\begin{equation}\label{eq:X-boundary-decomposition}
        X=\sin\theta\,\nu-\cos\theta\,J\nu,
\end{equation}
so
\begin{equation}\label{eq:Xperp-boundary}
        X^\perp=-\cos\theta\,J\nu
\end{equation}
and
\begin{equation}\label{eq:H-boundary}
        H=\varepsilon\cos\theta\,J\nu.
\end{equation}
\end{lemma}

\begin{proof}
Differentiate the capillary identity
$$
        \nu=\sin\theta\,X+\cos\theta\,JX
$$
in a tangential direction $U\in T\Sigma$.  Since $U(\theta)=0$ and $D_UX=U$,
$$
        D_U\nu=\sin\theta\,U+\cos\theta\,JU.
$$
Because $M$ is Lagrangian, $JU$ is normal to $M$.  Taking the tangential and normal components gives \eqref{eq:boundary-nabla-nu} and \eqref{eq:boundary-B-Unu}.

Apply $J$ to the capillary identity:
$$
        J\nu=\sin\theta\,JX-\cos\theta\,X.
$$
Solving this together with the original identity for $X$ gives \eqref{eq:X-boundary-decomposition}.  Since $\nu$ is tangent to $M$ and $J\nu$ is normal, the normal projection is \eqref{eq:Xperp-boundary}.  Equation \eqref{eq:intro-self-similar} then yields \eqref{eq:H-boundary}.
\end{proof}

The cubic symmetry gives the components of $B$ that are needed for the boundary mean-curvature calculation.

\begin{lemma}
For $U,V\in T\Sigma$,
\begin{equation}\label{eq:B-UV-Jnu}
        \langle B(U,V),J\nu\rangle
        =\cos\theta\,\langle U,V\rangle.
\end{equation}
For $U\in T\Sigma$,
\begin{equation}\label{eq:B-nunu-JU}
        \langle B(\nu,\nu),JU\rangle=0.
\end{equation}
\end{lemma}

\begin{proof}
By the total symmetry of the cubic form and \eqref{eq:boundary-B-Unu},
$$
\begin{aligned}
        \langle B(U,V),J\nu\rangle
        &=\langle B(U,\nu),JV\rangle\\
        &=\cos\theta\,\langle JU,JV\rangle\\
        &=\cos\theta\,\langle U,V\rangle.
\end{aligned}
$$
This proves \eqref{eq:B-UV-Jnu}.  Similarly,
$$
\begin{aligned}
        \langle B(\nu,\nu),JU\rangle
        &=\langle B(\nu,U),J\nu\rangle\\
        &=\cos\theta\,\langle JU,J\nu\rangle
        \\&=0,
\end{aligned}
$$
because $U\perp\nu$.
\end{proof}

We now prove the high-dimensional extension of the two-dimensional great-circle lemma of Luo--Sun \cite{LuoSun2021} for minimal Lagrangian surfaces.  In dimension two, a connected one-dimensional minimal submanifold of the round sphere is a geodesic, so the result reduces exactly to the boundary great-circle statement.

\begin{theorem}\label{thm:boundary-minimal}
Let $X:M^n\to\C^n$ satisfy \eqref{eq:intro-self-similar}, and let $\Sigma\subset\partial M$ be a Legendrian capillary boundary component with constant contact angle.  Then the immersion
$$
        X|_\Sigma:\Sigma^{n-1}\longrightarrow\Sph^{2n-1}
$$
is minimal Legendrian.
\end{theorem}

\begin{proof}
The boundary is Legendrian by assumption.  It remains to prove minimality in the unit sphere.  Choose a local orthonormal frame
$$
        e_1,\ldots,e_{n-1}
$$
of $T\Sigma$.  The normal space of $\Sigma$ inside $\Sph^{2n-1}$ is spanned by
\begin{equation}\label{eq:spherical-normal-frame}
        JX,
        \qquad
        Je_1,\ldots,Je_{n-1}.
\end{equation}
Indeed, all these vectors are tangent to the sphere and orthogonal to $T\Sigma$.  They are orthonormal, and their number equals the codimension of $\Sigma$ in $\Sph^{2n-1}$.

Let $B^\Sigma_{\Sph}$ denote the second fundamental form of $\Sigma$ in the unit sphere and $H^\Sigma_{\Sph}$ its trace.  We compute the components of $H^\Sigma_{\Sph}$ in the frame \eqref{eq:spherical-normal-frame}.

First consider the $JX$-component.  Since $\Sigma$ is Legendrian,
$$
        \langle e_b,JX\rangle=0.
$$
Differentiating in the direction $e_a$ gives
$$
        0
        =\langle D_{e_a}e_b,JX\rangle
          +\langle e_b,Je_a\rangle.
$$
The second term vanishes because $M$ is Lagrangian.  Hence
$$
        \langle D_{e_a}e_b,JX\rangle=0.
$$
The radial second fundamental form of the sphere is a multiple of $X$, which is orthogonal to $JX$.  Therefore
\begin{equation}\label{eq:spherical-B-JX}
        \langle B^\Sigma_{\Sph}(e_a,e_b),JX\rangle=0.
\end{equation}

Next fix $c\in\{1,\ldots,n-1\}$.  Since $Je_c$ is normal to $M$,
$$
        \langle B^\Sigma_{\Sph}(e_a,e_a),Je_c\rangle
        =\langle B(e_a,e_a),Je_c\rangle.
$$
Taking the trace and using \eqref{eq:H-boundary},
$$
\begin{aligned}
        \sum_{a=1}^{n-1}
        \langle B(e_a,e_a),Je_c\rangle
        &=\langle H-B(\nu,\nu),Je_c\rangle\\
        &=\varepsilon\cos\theta\,\langle J\nu,Je_c\rangle
          -\langle B(\nu,\nu),Je_c\rangle.
\end{aligned}
$$
The first term is zero because $\nu\perp e_c$, and the second is zero by \eqref{eq:B-nunu-JU}.  Thus
\begin{equation}\label{eq:spherical-H-Jec}
        \langle H^\Sigma_{\Sph},Je_c\rangle=0
        \qquad\text{for every }c.
\end{equation}
Equations \eqref{eq:spherical-B-JX} and \eqref{eq:spherical-H-Jec} show that every component of $H^\Sigma_{\Sph}$ vanishes.  Hence $\Sigma$ is minimal in $\Sph^{2n-1}$.
\end{proof}

\begin{remark}
When orientations are available, Theorem~\ref{thm:boundary-minimal} also follows immediately from Lemma~\ref{lem:phase-relation}.  On $T\Sigma$, the self-similar submanifold phase identity \eqref{eq:d-beta-lambda} and the Legendrian condition \eqref{eq:lambda-boundary-zero} give $d\beta_M=0$.  Since $d\theta=0$, equation \eqref{eq:d-phase-relation} gives $d\beta_\Sigma=0$.  Lemma~\ref{lem:legendrian-minimal-phase} then implies minimality.  The tensor proof above is preferable for the main theorem because it is local and requires no orientability.
\end{remark}

\begin{corollary}
If $n=2$, every boundary component is a possibly multiply covered Legendrian great circle in $\Sph^3$.
\end{corollary}

\begin{proof}
By Theorem~\ref{thm:boundary-minimal}, the boundary curve is a one-dimensional minimal submanifold of the round sphere.  Hence it is a geodesic.  Every closed geodesic in $\Sph^3$ is a great circle, possibly covered finitely many times.
\end{proof}

\section{Local product structure and intrinsic splitting}\label{sec:rank-one}

This section has three purposes. We first derive the higher-dimensional Anciaux equations for $H+\varepsilon X^\perp=0$. Using the minimal Legendrian geometry of a boundary component  and one-sided Cauchy uniqueness, we then identify the immersion \(X\) with the corresponding product model on a boundary collar. The resulting rank-one identity extends to \(M^\circ\) by real analyticity  and to $M$ by continuity.  Finally, we show intrinsically that, wherever the associated line field is nondegenerate, this identity forces the local Anciaux splitting and its profile equation.

\subsection{Product models and capillary data}
\label{subsec:product-models}

Let \(\psi:N^{n-1}\longrightarrow\Sph^{2n-1}\) be a Legendrian immersion, and let
\(\gamma:I\longrightarrow\C^*\) be a planar curve parametrized by arclength. Write
\(\gamma=r\ee^{i\phi}\). Since \(|\gamma'|=1\), there exists a locally defined
tangent angle \(\alpha\) such that \(\gamma'=\ee^{i\alpha}\).

Differentiating \(\gamma=r\ee^{i\phi}\), we obtain
\[
\gamma'=\ee^{i\phi}\bigl(r'+ir\phi'\bigr).
\]
It follows that
\[
r'+ir\phi'
=\ee^{-i\phi}\gamma'
=\ee^{i(\alpha-\phi)}.
\]
Set \(\Theta:=\alpha-\phi\). Then
\[
r'+ir\phi'
=\ee^{i\Theta}
=\cos\Theta+i\sin\Theta.
\]
Comparing the real and imaginary parts gives
\begin{equation}\label{eq:polar-basic}
r'=\cos\Theta,
\qquad
r\phi'=\sin\Theta.
\end{equation}
Thus \(\Theta\) is the oriented angle from the radial direction
\(\ee^{i\phi}\) to the unit tangent vector \(\gamma'\).

Define
\begin{equation}\label{eq:anciaux-product}
\begin{aligned}
        F:I\times N&\longrightarrow\C^n,\\
        (s,p)&\longmapsto\gamma(s)\psi(p).
\end{aligned}
\end{equation}
The following proposition is the self-similar product calculation underlying Anciaux's construction \cite{Anciaux2006}.  We include the complete calculation because it will later be recovered intrinsically from the rank-one identity.

\begin{proposition}\label{prop:anciaux-product}
Suppose that $\psi:N^{n-1}\to\Sph^{2n-1}$ is a minimal Legendrian immersion.  Then $F$ in \eqref{eq:anciaux-product} is Lagrangian.  Its induced metric is
\begin{equation}\label{eq:product-metric}
        g_F=ds^2+r(s)^2g_N.
\end{equation}
Moreover, $F$ satisfies
$$
        H_F+\varepsilon F^\perp=0
$$
if and only if
\begin{equation}\label{eq:anciaux-system}
        r'=\cos\Theta,
        \qquad
        \Theta'=\left(\varepsilon r-\frac nr\right)\sin\Theta.
\end{equation}
The system has the first integral
\begin{equation}\label{eq:anciaux-first-integral}
        E=r^n\ee^{-\varepsilon r^2/2}\sin\Theta.
\end{equation}
\end{proposition}

\begin{proof}
Let $v,w\in T_pN$. The tangent vectors of $F$ are
\begin{equation}\label{eq:product-tangent-vectors}
F_s=\gamma'\psi,
\qquad
F_*v=\gamma\psi_*v.
\end{equation}
Since $\psi$ takes values in the sphere,
$\langle\psi,\psi_*v\rangle=0$.
Since it is Legendrian,
$\langle J\psi,\psi_*v\rangle=0$
and $\omega(\psi_*v,\psi_*w)=0$.
It follows from \eqref{eq:product-tangent-vectors} that
$\omega(F_s,F_*v)=0$ and
$\omega(F_*v,F_*w)=0$.
Thus $F$ is Lagrangian. The same orthogonality relations give
$|F_s|=1$, $\langle F_s,F_*v\rangle=0$, and
$\langle F_*v,F_*w\rangle=r^2\langle v,w\rangle$,
which proves \eqref{eq:product-metric}.

Choose a local orthonormal frame $v_1,\ldots,v_{n-1}$ on $N$, geodesic at the point under consideration, and set
\begin{equation*}
e_0=F_s=\ee^{i\alpha}\psi,
\qquad
e_a=\frac1rF_*v_a=\ee^{i\phi}\psi_*v_a.
\end{equation*}
This is an orthonormal tangent frame for $F$.
We first compute the $e_0$-direction. Since $\gamma'=\ee^{i\alpha}$,
we have
\begin{equation}\label{eq:D-e0-e0-product}
D_{e_0}e_0=\alpha'Je_0.
\end{equation}

For the directions tangent to $N$, note that the Euclidean and spherical connections are related by
\[
D^{\C^n}_{v_a}(\psi_*v_a)
=\nabla^{\Sph}_{v_a}(\psi_*v_a)-\psi
\]
after choosing $\nabla^N_{v_a}v_a=0$ at the point. Taking the trace in $a$, minimality of $\psi$ in the sphere gives
\begin{equation}
\sum_{a=1}^{n-1}D^{\C^n}_{v_a}(\psi_*v_a)
=-(n-1)\psi.
\end{equation}
Because differentiation in the unit vector direction $e_a$ introduces a factor $1/r$,
\begin{equation}\label{eq:trace-N-directions}
\sum_{a=1}^{n-1}D_{e_a}e_a
=-\frac{n-1}{r}\ee^{i\phi}\psi.
\end{equation}
The vector $\ee^{i\phi}\psi$ lies in the real two-plane generated by $e_0$ and $Je_0$.
From $e_0=\ee^{i(\phi+\Theta)}\psi$,
\begin{equation}
\ee^{i\phi}\psi
=\cos\Theta\,e_0-\sin\Theta\,Je_0.
\end{equation}
Thus the normal component of \eqref{eq:trace-N-directions} is
\begin{equation}\label{eq:normal-trace-N}
\left(\sum_{a=1}^{n-1}D_{e_a}e_a\right)^\perp
=\frac{n-1}{r}\sin\Theta\,Je_0.
\end{equation}

The trace has no component in the normal directions $Je_a$: such components are precisely the spherical mean-curvature components of the link, which vanish by minimality.
Combining \eqref{eq:D-e0-e0-product} and \eqref{eq:normal-trace-N},
\begin{equation*}
H_F
=\left(\alpha'+\frac{n-1}{r}\sin\Theta\right)Je_0.
\end{equation*}
On the other hand,
$F=r\ee^{i\phi}\psi
=r\cos\Theta\,e_0-r\sin\Theta\,Je_0$,
so
\begin{equation*}
F^\perp=-r\sin\Theta\,Je_0.
\end{equation*}
Therefore $H_F+\varepsilon F^\perp=0$ is equivalent to
\begin{equation}\label{eq:alpha-prime-product}
\alpha'
=\left(\varepsilon r-\frac{n-1}{r}\right)\sin\Theta.
\end{equation}
Subtracting $\phi'=r^{-1}\sin\Theta$ from \eqref{eq:alpha-prime-product} gives the second equation in \eqref{eq:anciaux-system}, and the first is already contained in \eqref{eq:polar-basic}.

Finally, differentiate
$E=r^n\ee^{-\varepsilon r^2/2}\sin\Theta$.
Using \eqref{eq:anciaux-system},
\begin{align*}
E'
&=\ee^{-\varepsilon r^2/2}
\left((nr^{n-1}-\varepsilon r^{n+1})r'\sin\Theta
+r^n\cos\Theta\,\Theta'\right)\\
&=r^{n-1}\ee^{-\varepsilon r^2/2}\cos\Theta\sin\Theta
\left(n-\varepsilon r^2+\varepsilon r^2-n\right)\\
&=0.
\end{align*}
No division by $\sin\Theta$ is used, so the identity remains valid at its zeros.
\end{proof}

\begin{proposition}\label{prop:product-boundary-angles}
Let $F(s,p)=\gamma(s)\psi(p)$ satisfy Proposition~\ref{prop:anciaux-product} on a compact interval $[s_0,s_1]$, and assume that $N$ is connected.  Suppose
\begin{equation}\label{eq:product-unit-ball-segment}
        r(s_0)=r(s_1)=1,
        \qquad
        0<r(s)<1\quad(s_0<s<s_1).
\end{equation}
Then the two boundary components
\[
        \{s_0\}\times N,
        \qquad
        \{s_1\}\times N
\]
of the parameter manifold are mapped by $F$ to constant-angle Legendrian capillary boundary immersions.  With outward conormal $-F_s$ at $s_0$ and $F_s$ at $s_1$, the contact angles satisfy
\begin{equation}\label{eq:product-angle-components}
\begin{aligned}
        \sin\theta_0&=-\cos\Theta(s_0),
        &\qquad \cos\theta_0&=-\sin\Theta(s_0),\\
        \sin\theta_1&= \cos\Theta(s_1),
        &\qquad \cos\theta_1&= \sin\Theta(s_1).
\end{aligned}
\end{equation}
In particular,
\begin{equation}\label{eq:product-supplementary}
        \cos\theta_0+\cos\theta_1=0,
        \qquad
        \theta_0+\theta_1=\pi.
\end{equation}
\end{proposition}

\begin{proof}
At an endpoint $s=s_j$, the position vector has length one and
$$
        F_s
        =\ee^{i\alpha}\psi
        =\ee^{i\Theta}F
        =\cos\Theta\,F+\sin\Theta\,JF.
$$
The boundary immersion $p\mapsto\ee^{i\phi(s_j)}\psi(p)$ is Legendrian because
unitary multiplication preserves the standard contact structure.  We first
check the signs of the radial derivatives.  Since $r(s)<1$ immediately to the
right of $s_0$, one has $r'(s_0)\le0$.  Equality is impossible: if
$r'(s_0)=0$, then $|\sin\Theta(s_0)|=1$ and the ODE gives
$$
        r''(s_0)=\frac n{r(s_0)}-\varepsilon r(s_0)=n-\varepsilon>0,
$$
so $r(s)>1$ for $s>s_0$ sufficiently small, contradicting
\eqref{eq:product-unit-ball-segment}.  Thus $r'(s_0)<0$.
Similarly, $r'(s_1)>0$.  Consequently the
coefficients of $F$ in the outward conormals $-F_s$ and $F_s$ are positive,
so they determine unique contact angles in $(0,\pi)$.  Comparing these
conormals with $\nu=\sin\theta F+\cos\theta JF$ gives
\eqref{eq:product-angle-components}.

At both endpoints, the first integral \eqref{eq:anciaux-first-integral} gives
$$
        \ee^{-\varepsilon/2}\sin\Theta(s_0)
        =E
        =\ee^{-\varepsilon/2}\sin\Theta(s_1).
$$
Hence $$\sin\Theta(s_0)=\sin\Theta(s_1),$$ and the cosine formulas in
\eqref{eq:product-angle-components} yield
$$\cos\theta_0+\cos\theta_1=0.$$  Since
$\theta_0,\theta_1\in(0,\pi)$ and $\cos$ is strictly decreasing on
$[0,\pi]$, this is equivalent to $\theta_1=\pi-\theta_0$.  This proves
\eqref{eq:product-supplementary}.
\end{proof}

\subsection{Boundary collars and analytic continuation}

Fix a connected boundary component $\Sigma_0\subset\partial M$, and let $\theta_0$ be its contact angle.  By Theorem~\ref{thm:boundary-minimal},
\begin{equation}
        \psi=X|_{\Sigma_0}:\Sigma_0\longrightarrow\Sph^{2n-1}
\end{equation}
is a compact minimal Legendrian immersion.

Choose the inward unit conormal
$$
        e_0^{\mathrm{in}}=-\nu.
$$
Along $\Sigma_0$,
\begin{equation}\label{eq:initial-conormal}
        e_0^{\mathrm{in}}
        =-\sin\theta_0\,X-\cos\theta_0\,JX.
\end{equation}
Set
\begin{equation}\label{eq:anciaux-initial-polar-data}
        r(0)=1,
        \qquad
        \Theta(0)=\frac{3\pi}{2}-\theta_0,
        \qquad
        \phi(0)=0.
\end{equation}
Then
$\cos\Theta(0)=-\sin\theta_0$ and
$\sin\Theta(0)=-\cos\theta_0$.  Solve the smooth system
\begin{equation*}
        r'=\cos\Theta,
        \qquad
        \Theta'=\left(\varepsilon r-\frac nr\right)\sin\Theta,
        \qquad
        \phi'=\frac{\sin\Theta}{r}
\end{equation*}
with the data \eqref{eq:anciaux-initial-polar-data}, and put
$\gamma=r\ee^{i\phi}$.  The right-hand side is smooth on $r>0$, so the
classical local existence, uniqueness, and smooth-dependence theorem
\cite[Chapter~1]{CoddingtonLevinson1955} gives a unique solution on
$[0,\delta)$ for some $\delta>0$.  It satisfies
\begin{equation}\label{eq:anciaux-initial-data}
        \gamma(0)=1,
        \qquad
        \gamma'(0)=-(\sin\theta_0+i\cos\theta_0).
\end{equation}

\begin{lemma}\label{lem:anciaux-collar}
There exist $\delta>0$ and a collar diffeomorphism
$$
        \Psi:[0,\delta)\times\Sigma_0\longrightarrow U\subset M,
$$
with $\Psi(0,p)=p$ such that
\begin{equation}\label{eq:collar-product}
        X\circ\Psi(s,p)=\gamma(s)\psi(p),
\end{equation}
where $\gamma$ is the self-similar curve determined by \eqref{eq:anciaux-initial-data}.
\end{lemma}

\begin{proof}
Using the metric induced by $X$, choose $\delta_0>0$ so that
$$
        \Psi_0(s,p)=\exp_p^M(-s\nu_p)
$$
defines a smooth inward normal geodesic collar
$$
        \Psi_0:[0,\delta_0)\times\Sigma_0\longrightarrow U_0\subset M.
$$
Then
$$
        \Psi_0(0,p)=p,
        \qquad
        d\Psi_0(\partial_s)|_{s=0}=-\nu.
$$
Define
$$
        F(s,p)=\gamma(s)\psi(p).
$$
By Proposition~\ref{prop:anciaux-product}, $F$ satisfies $H_F+\varepsilon F^\perp=0$.  At $s=0$,
$$
        F(0,p)=\psi(p)=X(p).
$$
For $W\in T_p\Sigma_0$,
$$
        F_*W(0,p)=\psi_*W=X_*W.
$$
Finally, by \eqref{eq:anciaux-initial-data} and \eqref{eq:initial-conormal},
\begin{equation*}
\begin{aligned}
     F_s(0,p)
        &=-(\sin\theta_0+i\cos\theta_0)\psi(p)\\
        &=-\sin\theta_0\,X(p)-\cos\theta_0\,JX(p)\\
        &=X_*(-\nu).
\end{aligned}
\end{equation*}
       
Thus $F$ and $X\circ\Psi_0$ have the same boundary parametrization 
and the same inward unit conormal derivative along $\{0\}\times\Sigma_0$.   Lemma~\ref{lem:cauchy}, applied
with the identity map on $\Sigma_0$,  
therefore yields a one-sided neighborhood and 
a collar diffeomorphism $\Psi:[0,\delta)\times\Sigma_0\to U$ fixing $\{0\}\times \Sigma_0$, such that $X\circ \Psi=F$. This is \eqref{eq:collar-product}.
\end{proof}

Define
\begin{equation*}
        \rho=\frac{|X|^2}{2}.
\end{equation*}
On the product collar \eqref{eq:collar-product},
$$
        \rho=\frac{r^2}{2}
$$
and hence
\begin{equation}\label{eq:drho-collar}
        d\rho=rr'\,ds.
\end{equation}
Moreover, the Legendrian condition gives $\lambda(W)=0$ for $W\in T\Sigma_0$, while

\begin{equation*}
\begin{aligned}
        \lambda(\partial_s)
        &=\langle J(\gamma\psi),\gamma'\psi\rangle\\
        &=\langle i\gamma,\gamma'\rangle\\
        &=r^2\phi'\\
        &=r\sin\Theta.
\end{aligned}
\end{equation*}

Thus
\begin{equation}\label{eq:lambda-collar}
        \lambda=r\sin\Theta\,ds.
\end{equation}
Equations \eqref{eq:drho-collar} and \eqref{eq:lambda-collar} imply
\begin{equation}
        d\rho\wedge\lambda=0
\end{equation}
on a nonempty open subset of $M^\circ$.

\begin{proposition}\label{prop:global-rank-one}
Under the assumptions of Theorem~\ref{thm:boundary-components},
\begin{equation}
        d\rho\wedge\lambda\equiv0
        \qquad\ \text{on }M.
\end{equation}
\end{proposition}

\begin{proof}
Equip $M^\circ$ with the real-analytic atlas supplied by
Lemma~\ref{lem:analyticity}.  In these analytic coordinates, $X$ is real
analytic.  Since
$$
        \rho=\frac12\sum_{\alpha=1}^{2n}(X^\alpha)^2
$$
and
$$
        \lambda(Y)=\langle JX,dX(Y)\rangle,
$$
the coordinate coefficients of $\rho$, $d\rho$, and $\lambda$ are real
analytic.  Consequently every coefficient of the two-form
$$
        \Xi:=d\rho\wedge\lambda
$$
is real analytic on $M^\circ$.

The collar supplied by Lemma~\ref{lem:anciaux-collar} contains an open subset
$(0,\delta)\times\Sigma_0$ of the interior.  Equations
\eqref{eq:drho-collar} and \eqref{eq:lambda-collar} show that both one-forms are
multiples of $ds$ there; hence $\Xi=0$ on that nonempty open set.

We recall why the identity theorem can be applied globally.  The interior of a
connected smooth manifold with nonempty boundary is connected: any path in
$M$ joining two interior points can be pushed slightly away from the boundary
inside a collar.  Cover such an interior path by finitely many connected
analytic coordinate balls with consecutive overlaps.  The coefficients of
$\Xi$ vanish in the first ball.  The real-analytic identity theorem
\cite[Section~1.2]{KrantzParks2002} propagates this vanishing across each
nonempty overlap and therefore along the whole chain of balls.  Since the
terminal point was arbitrary,
$$
        \Xi=d\rho\wedge\lambda=0
        \qquad\ \text{on }M^\circ.
$$
Finally, $d\rho$ and $\lambda$ extend smoothly to $\partial M$.  Taking limits
from the interior gives the same identity at every boundary point.
\end{proof}

\subsection{Intrinsic splitting and the profile equation}

We now work intrinsically on a connected open set on which the rank-one identity holds and the relevant line field is nondegenerate.  Every formula in this section is derived directly from
$$
        d\rho\wedge\lambda=0,
        \qquad
        \rho=\frac{|X|^2}{2},
$$
together with the Lagrangian and self-similar equations.  No ambient symmetry is assumed.

Define
\begin{equation*}
        U=\nabla\rho=X^\top,
        \qquad
        V=\lambda^\sharp=(JX)^\top.
\end{equation*}
The first identity follows from
$$
        d\rho(Y)=\langle X,Y\rangle
        =\langle X^\top,Y\rangle.
$$
The second is the metric dual of \eqref{eq:lambda-evaluation}.

\begin{lemma}\label{lem:X-U-V}
At every point of $M$,
\begin{equation*}
        X=U-JV.
\end{equation*}
Consequently,
\begin{equation}\label{eq:r-UV}
        |X|^2=|U|^2+|V|^2.
\end{equation}
\end{lemma}

\begin{proof}
Write
$$
        X=U+X^\perp.
$$
Since $M$ is Lagrangian, there is a tangent vector $Z$ such that
$$
        X^\perp=JZ.
$$
Then
$$
        JX=JU-Z.
$$
The vector $JU$ is normal, so the tangential component of $JX$ is $-Z$.  By definition this tangential component is $V$.  Hence $Z=-V$, and therefore
$$
        X=U+JZ=U-JV.
$$
The two vectors $U$ and $JV$ are orthogonal because one is tangent and the other normal.  Taking norms gives \eqref{eq:r-UV}.
\end{proof}

By Proposition~\ref{prop:global-rank-one}, the one-forms dual to $U$ and $V$ have zero wedge product.  Thus $U$ and $V$ are linearly dependent at every point.  If $X\ne0$, Lemma~\ref{lem:X-U-V} shows that they cannot both vanish.

Let $\mathcal O\subset M$ be a connected open set on which $X\ne0$ and on which the common line generated by $U,V$ is orientable.  Choose a unit vector field $e_0$ spanning that line and write
\begin{equation*}
        U=a e_0,
        \qquad
        V=b e_0.
\end{equation*}
Let
\begin{equation*}
        \eta=e_0^\flat,
        \qquad
        \Dcal=\ker\eta=e_0^\perp,
\end{equation*}
where $e_0^\flat$ denotes the dual one-form of $e_0$.
By Lemma~\ref{lem:X-U-V},
\begin{equation}\label{eq:X-a-b}
        X=a e_0-bJe_0.
\end{equation}
Put
\begin{equation}\label{eq:r-a-b}
        r=|X|=\sqrt{a^2+b^2}.
\end{equation}
The two closed one-forms take the simple form
\begin{equation}\label{eq:closed-forms-a-b}
        d\rho=a\eta,
        \qquad
        \lambda=b\eta.
\end{equation}

\begin{proposition}\label{prop:eta-closed}
On $\mathcal O$,
\begin{equation}\label{eq:eta-closed}
        d\eta=0.
\end{equation}
Moreover, for every $W\in\Dcal$,
\begin{equation}
        W(a)=W(b)=W(r)=0,
\end{equation}
and
\begin{equation}
        \nabla_{e_0}e_0=0.
\end{equation}
In particular, $\Dcal$ is integrable, and the integral curves of $e_0$ are unit-speed geodesics in $M$.
\end{proposition}

\begin{proof}
Since $d(d\rho)=0$ and $d\lambda=0$, equation \eqref{eq:closed-forms-a-b} gives
\begin{equation}\label{eq:closed-equations-a-b}
        da\wedge\eta+a\,d\eta=0,
        \qquad
        db\wedge\eta+b\,d\eta=0.
\end{equation}
Let $W,Z\in\Dcal$.  Evaluating the two equations on $(W,Z)$,
$$
        a\,d\eta(W,Z)=0,
        \qquad
        b\,d\eta(W,Z)=0.
$$
Because $a$ and $b$ do not vanish simultaneously, we obtain
\begin{equation}\label{eq:deta-D-D-zero}
        d\eta(W,Z)=0.
\end{equation}

Next evaluate \eqref{eq:closed-equations-a-b} on $(e_0,W)$.  Since $$\eta(e_0)=1, \eta(W)=0,$$
we have
\begin{equation}\label{eq:W-a-b-c}
        W(a)=a\,d\eta(e_0,W),
        \qquad
        W(b)=b\,d\eta(e_0,W).
\end{equation}
On the other hand, $d\rho(W)=0$, so
$$
        W(r^2)=2W(\rho)=0.
$$
Using \eqref{eq:r-a-b} and \eqref{eq:W-a-b-c},
$$
\begin{aligned}
        0
        &=2aW(a)+2bW(b)\\
        &=2(a^2+b^2)d\eta(e_0,W)\\
        &=2r^2d\eta(e_0,W).
\end{aligned}
$$
Since $r>0$,
\begin{equation}\label{eq:deta-e0-W-zero}
        d\eta(e_0,W)=0.
\end{equation}
Substituting this back into \eqref{eq:W-a-b-c} gives $W(a)=W(b)=0$, and therefore $W(r)=0$.  Together with \eqref{eq:deta-D-D-zero}, equation \eqref{eq:deta-e0-W-zero} proves \eqref{eq:eta-closed}.

The Frobenius theorem now gives integrability of $\Dcal=\ker\eta$.  To prove geodesicity, let $W\in\Dcal$.  Since $\eta=e_0^\flat$,
$$
\begin{aligned}
        d\eta(e_0,W)
        &=e_0\langle e_0,W\rangle
          -W\langle e_0,e_0\rangle
          -\langle e_0,[e_0,W]\rangle\\
        &=-\langle e_0,\nabla_{e_0}W-\nabla_We_0\rangle\\
        &=\langle\nabla_{e_0}e_0,W\rangle.
\end{aligned}
$$
Thus \eqref{eq:deta-e0-W-zero} implies that $\nabla_{e_0}e_0$ is orthogonal to $\Dcal$.  It is also orthogonal to $e_0$, because $e_0$ has unit length.  Hence $\nabla_{e_0}e_0=0$.
\end{proof}

\begin{remark}
Locally, $d\eta=0$ implies the existence of a coordinate $s$ with $\eta=ds$.  Since $\eta(e_0)=1$, one has $e_0=\partial_s$, while the leaves of $\Dcal$ are the level sets of $s$.  The global proof in Section~\ref{sec:global-classification} constructs this coordinate by the flow of $e_0$, so no global exactness of $\eta$ is assumed here.
\end{remark}

\begin{proposition}
For every $W\in\Dcal$,
\begin{equation}\label{eq:D-W-e0}
        D_We_0
        =\frac{a}{r^2}W+\frac{b}{r^2}JW.
\end{equation}
Equivalently,
\begin{equation}\label{eq:nabla-B-W-e0}
        \nabla_We_0=\frac{a}{r^2}W,
        \qquad
        B(W,e_0)=\frac{b}{r^2}JW.
\end{equation}
Furthermore, there is a function $\kappa$ such that
\begin{equation}\label{eq:B-e0-e0-kappa}
        B(e_0,e_0)=\kappa Je_0,
\end{equation}
and the self-similar equation forces
\begin{equation}\label{eq:kappa-formula}
        \kappa=b\left(\varepsilon-\frac{n-1}{r^2}\right).
\end{equation}
\end{proposition}

\begin{proof}
Differentiate \eqref{eq:X-a-b} in a direction $W\in\Dcal$.  Proposition~\ref{prop:eta-closed} gives $W(a)=W(b)=0$, while $D_WX=W$.  Hence
\begin{equation}\label{eq:linear-equation-DWe0}
        W=aD_We_0-bJD_We_0.
\end{equation}
As an endomorphism of the ambient vector space,
$$
        (a\Id+bJ)(a\Id-bJ)
        =(a^2+b^2)\Id=r^2\Id.
$$
Applying $(a\Id+bJ)/r^2$ to \eqref{eq:linear-equation-DWe0} gives \eqref{eq:D-W-e0}.  Since $W$ is tangent and $JW$ is normal, the tangential and normal components are exactly \eqref{eq:nabla-B-W-e0}.

We next determine $B(e_0,e_0)$.  Let $W\in\Dcal$.  By the symmetry of the cubic form and \eqref{eq:nabla-B-W-e0},
$$
\begin{aligned}
        \langle B(e_0,e_0),JW\rangle
        &=\langle B(e_0,W),Je_0\rangle\\
        &=\frac{b}{r^2}\langle JW,Je_0\rangle=0.
\end{aligned}
$$
The normal space is $JTM$, and its orthogonal decomposition is
$$
        JTM=\R Je_0\oplus J\Dcal.
$$
The preceding identity shows that the $J\Dcal$-component of $B(e_0,e_0)$ vanishes, proving \eqref{eq:B-e0-e0-kappa}.

From \eqref{eq:X-a-b},
\begin{equation*}
        X^\perp=-bJe_0.
\end{equation*}
Thus \eqref{eq:intro-self-similar} gives
\begin{equation}\label{eq:H-b}
        H=\varepsilon bJe_0.
\end{equation}
Choose an orthonormal frame $e_1,\ldots,e_{n-1}$ of $\Dcal$.  By cubic symmetry and \eqref{eq:nabla-B-W-e0},
$$
\begin{aligned}
        \langle B(e_a,e_a),Je_0\rangle
        &=\langle B(e_a,e_0),Je_a\rangle\\
        &=\frac{b}{r^2}.
\end{aligned}
$$
Taking the $Je_0$-component of the trace identity \eqref{eq:H-b},
$$
        \kappa+(n-1)\frac{b}{r^2}=\varepsilon b.
$$
Solving for $\kappa$ gives \eqref{eq:kappa-formula}.
\end{proof}

Let a prime denote differentiation along $e_0$.  Since the integral curves of $e_0$ are unit speed, this agrees with ordinary arclength differentiation along those curves.

\begin{proposition}\label{prop:intrinsic-anciaux}
On $\mathcal O$,
\begin{equation}\label{eq:a-b-derivatives}
        a'=1-b\kappa,
        \qquad
        b'=a\kappa,
\end{equation}
where $\kappa$ is given by \eqref{eq:kappa-formula}.  Define $\Theta$ locally by
\begin{equation}\label{eq:theta-a-b}
        a=r\cos\Theta,
        \qquad
        b=r\sin\Theta.
\end{equation}
Then
\begin{equation}\label{eq:intrinsic-anciaux-system}
        r'=\cos\Theta,
        \qquad
        \Theta'=\left(\varepsilon r-\frac nr\right)\sin\Theta.
\end{equation}
Moreover,
\begin{equation}\label{eq:intrinsic-first-integral}
        E=r^n\ee^{-\varepsilon r^2/2}\sin\Theta
        =r^{n-1}b\ee^{-\varepsilon r^2/2}
\end{equation}
is constant on every connected component of $\mathcal O$.
\end{proposition}

\begin{proof}
Differentiate \eqref{eq:X-a-b} in the $e_0$-direction.  By Proposition~\ref{prop:eta-closed}, $\nabla_{e_0}e_0=0$, and by \eqref{eq:B-e0-e0-kappa},
$D_{e_0}e_0=\kappa Je_0.$
Since $D_{e_0}X=e_0$,
$$
\begin{aligned}
        e_0
        &=D_{e_0}(a e_0-bJe_0)\\
        &=a'e_0+a\kappa Je_0-b'Je_0-bJ(\kappa Je_0)\\
        &=(a'+b\kappa)e_0+(a\kappa-b')Je_0.
\end{aligned}
$$
Comparing the $e_0$ and $Je_0$ components gives \eqref{eq:a-b-derivatives}.

Because $\rho=r^2/2$,
$rr'=d\rho(e_0)=a.$
Using \eqref{eq:theta-a-b},
\begin{equation*}
        r'=\frac ar=\cos\Theta.
\end{equation*}
Also,
$\Theta'=\frac{ab'-ba'}{r^2}.$
Substitute \eqref{eq:a-b-derivatives}:
$$
\begin{aligned}
        \Theta'
        &=\frac{a(a\kappa)-b(1-b\kappa)}{r^2}\\
        &=\kappa-\frac{b}{r^2}.
\end{aligned}
$$
Using \eqref{eq:kappa-formula} and $b=r\sin\Theta$,
$$
\begin{aligned}
        \Theta'
        &=b\left(\varepsilon-\frac{n-1}{r^2}\right)-\frac{b}{r^2}\\
        &=b\left(\varepsilon-\frac n{r^2}\right)\\
        &=\left(\varepsilon r-\frac nr\right)\sin\Theta.
\end{aligned}
$$
This proves \eqref{eq:intrinsic-anciaux-system}.

To prove constancy of $E$, first note from Proposition~\ref{prop:eta-closed} that $r$ and $b$ are constant in every $\Dcal$-direction.  It remains to differentiate along $e_0$.  Using the second expression in \eqref{eq:intrinsic-first-integral},
$$
\begin{aligned}
        E'
        &=\ee^{-\varepsilon r^2/2}\left((n-1)r^{n-2}r'b+r^{n-1}b'-\varepsilon r^nr'b\right)\\
        &=r^{n-2}\ee^{-\varepsilon r^2/2}
          \left((n-1)r'b+rb'-\varepsilon r^2r'b\right).
\end{aligned}
$$
Since $r'=a/r$, $b'=a\kappa$, and
$\kappa=b(\varepsilon-(n-1)/r^2)$, the bracket equals
$$
        (n-1)\frac{ab}{r}
        +ra b\left(\varepsilon-\frac{n-1}{r^2}\right)
        -\varepsilon rab=0.
$$
Thus $E'=0$, and $E$ is constant on the connected set.
\end{proof}

On the product collar, the angle $\Theta$ defined by \eqref{eq:theta-a-b} agrees with the profile angle introduced in Subsection~\ref{subsec:product-models}.

The intrinsic splitting formulas already imply that the normalized link is constant along the $e_0$-geodesics.

\begin{proposition}\label{prop:local-reconstruction}
Let $s$ be a local coordinate with $e_0=\partial_s$, and define $\phi$ by
\begin{equation}\label{eq:phi-prime-intrinsic}
        \phi'=\frac{b}{r^2}.
\end{equation}
Then
\begin{equation}\label{eq:normalized-link-parallel}
        D_{e_0}\left(\ee^{-i\phi}\frac{X}{r}\right)=0.
\end{equation}
Consequently, after choosing a leaf $N$ of $\Dcal$, there are local product coordinates in which
\begin{equation}\label{eq:local-product-reconstruction}
        X(s,p)=\gamma(s)\psi(p),
        \qquad
        \gamma(s)=r(s)\ee^{i\phi(s)},
\end{equation}
where $\psi:N\to\Sph^{2n-1}$ is a Legendrian immersion.  The curve $\gamma$ is arclength-parametrized and satisfies the self-similar system.
\end{proposition}

\begin{proof}
We first compute the derivative of $X/r$.  Since $D_{e_0}X=e_0$ and $r'=a/r$,
$$
\begin{aligned}
        D_{e_0}\left(\frac Xr\right)
        &=\frac1r e_0-\frac{r'}{r^2}X\\
        &=\frac1r e_0-\frac{a}{r^3}(a e_0-bJe_0)\\
        &=\frac{b^2}{r^3}e_0+\frac{ab}{r^3}Je_0.
\end{aligned}
$$
On the other hand,
$$
        J\left(\frac Xr\right)
        =\frac1r(aJe_0+b e_0).
$$
Therefore
\begin{equation}
        D_{e_0}\left(\frac Xr\right)
        =\frac{b}{r^2}J\left(\frac Xr\right).
\end{equation}
Upon differentiating $\ee^{-i\phi}X/r$ and using \eqref{eq:phi-prime-intrinsic}, the two terms cancel and give \eqref{eq:normalized-link-parallel}.

Because $d\eta=0$, choose a local leaf $N$ of $\Dcal$ and use the flow of
$e_0$ to define
$$
        \chi:(-\delta,\delta)\times N\longrightarrow\mathcal O,
        \qquad
        \chi(s,p)=\varphi_s(p).
$$
Cartan's formula gives $\mathcal L_{e_0}\eta=0$, so the flow preserves
$\Dcal$; hence $\chi_*\partial_s=e_0$ and
$\chi_*(TN)=\Dcal$.  Proposition~\ref{prop:eta-closed} also shows that
$b/r^2$ is constant along every $\Dcal$-leaf, so $\phi$ may be chosen as a
function of $s$ alone.  Equation \eqref{eq:normalized-link-parallel} then
shows that
$$
        \psi(p)=\ee^{-i\phi(s)}\frac{X(\chi(s,p))}{r(s)}
$$
is independent of $s$.  This proves \eqref{eq:local-product-reconstruction}.
It takes values in the unit sphere because $|X|=r$.  For $W\in T_pN$,
$$
        \psi_*W
        =\frac{\ee^{-i\phi(s)}}{r(s)}X_*\bigl(\chi_*W\bigr).
$$
Since $\chi$ is a local diffeomorphism, $X$ is an immersion, and $r(s)>0$,
this formula shows that $\psi_*W=0$ only if $W=0$.  Thus $\psi$ is an
immersion.  Moreover,
$$
        \langle J\psi,\psi_*W\rangle
        =\frac1{r^2}\langle JX,X_*\bigl(\chi_*W\bigr)\rangle
        =\frac1{r^2}\lambda(\chi_*W)=0,
$$
because $\chi_*W\in\Dcal$.  Hence $\psi$ is Legendrian.

Finally,
$$
\begin{aligned}
        \gamma'
        &=\ee^{i\phi}(r'+ir\phi')\\
        &=\ee^{i\phi}\left(\frac ar+i\frac br\right).
\end{aligned}
$$
Since $a^2+b^2=r^2$, $|\gamma'|=1$.  Proposition~\ref{prop:intrinsic-anciaux} gives the self-similar system.
\end{proof}

\section{Global classification}\label{sec:global-classification}

We now assume all hypotheses of Theorem~\ref{thm:boundary-components}.  Fix a connected boundary component $\Sigma_0$, let $\theta_0$ be its contact angle, and choose the self-similar Legendrian collar of Lemma~\ref{lem:anciaux-collar}.  On this collar we orient the rank-one line field by the inward unit conormal:
\begin{equation}\label{eq:global-e0-initial}
        e_0=-\nu
        \qquad\ \text{on }\Sigma_0.
\end{equation}
At $r=1$, the decomposition \eqref{eq:X-a-b} and the capillary condition give
\begin{equation}\label{eq:initial-a-b}
        a(0)=-\sin\theta_0,
        \qquad
        b(0)=-\cos\theta_0.
\end{equation}
Hence the first integral on the collar is
\begin{equation}\label{eq:E-contact-angle-initial}
        E=-\ee^{-\varepsilon/2}\cos\theta_0.
\end{equation}

\subsection{Global nondegeneracy and radial dynamics}
If $\theta_0=\frac{\pi}{2}$, then
\cite[Theorem~1.3]{GaoMaYao2026} implies that $X$ is a
diffeomorphism onto an equatorial Lagrangian $n$-disk. In particular,
$\partial M$ is connected. From now on, we assume
\begin{equation}\label{eq:nonfree-assumption}
\theta_0\neq\frac{\pi}{2}.
\end{equation}
Let
$$
        \Omega=\{q\in M:V(q)\ne0\}.
$$
The collar attached to $\Sigma_0$ lies in $\Omega$, because $b(0)=-\cos\theta_0\ne0$.  Let $\Omega_0$ be the connected component of $\Omega$ containing this collar.  On $\Omega_0$, 
define
\[
        e_0=\frac{\operatorname{sgn}(E)}{|V|}V.
\]
By \eqref{eq:initial-a-b} and \eqref{eq:E-contact-angle-initial},
\[
E=\ee^{-\varepsilon/2}b(0).
\]
Thus $\operatorname{sgn}(E)=\operatorname{sgn}(b(0))$, and this definition of $e_0$ agrees with the collar orientation \eqref{eq:global-e0-initial}. Hence all the formulas of Section~\ref{sec:rank-one} apply on $\Omega_0$, with $E$ given by  \eqref{eq:E-contact-angle-initial}.

\begin{proposition}\label{prop:global-nondegeneracy}
Under \eqref{eq:nonfree-assumption},
\begin{equation}\label{eq:V-nonzero-global}
        V=(JX)^\top\ne0
        \qquad\text{everywhere on }M.
\end{equation}
Consequently $X\ne0$, the unit vector field $e_0$ is globally defined, and the functions $a,b,r$ and the constant $E$ extend smoothly to all of $M$.
\end{proposition}

\begin{proof}
On $\Omega_0$ the oriented decomposition $V=b e_0$ is available and the
constant $E$ has the nonzero value in \eqref{eq:E-contact-angle-initial}.  The
second expression in \eqref{eq:intrinsic-first-integral} yields
\begin{equation*}
        b=E\frac{\ee^{\varepsilon r^2/2}}{r^{n-1}}.
\end{equation*}
In particular, $b$ has the fixed sign of $E$ throughout $\Omega_0$.  Since
$0<r\le1$,
$$
        \frac{\ee^{\varepsilon r^2/2}}{r^{n-1}}
        \ge c_\varepsilon,
        \qquad
        c_\varepsilon:=\exp\!\left(\frac{\min\{\varepsilon,0\}}2\right)>0.
$$
Indeed, $r^{-(n-1)}\ge1$; if $\varepsilon\ge0$, the exponential is at least one,
whereas if $\varepsilon<0$, then $\varepsilon r^2/2\ge\varepsilon/2$.  Hence
\begin{equation}\label{eq:b-lower-bound-E}
        |V|=|b|\ge c_\varepsilon|E|>0
        \qquad\ \text{on }\Omega_0.
\end{equation}

We prove that $\Omega_0$ is closed in $M$.  Let $q_j\in\Omega_0$ converge to
$q\in M$.  Continuity of $V$ and \eqref{eq:b-lower-bound-E} imply
$$
        |V(q)|\ge c_\varepsilon|E|>0,
$$
so $q\in\Omega$.  Choose a connected coordinate neighborhood
$U\Subset\Omega$ of $q$.  For all sufficiently large $j$, $q_j\in U$.  Since
$U$ is connected and meets $\Omega_0$, it is contained in the connected
component $\Omega_0$ of $\Omega$.  Thus $q\in\Omega_0$.  Therefore $\Omega_0$
is both open and closed in the connected manifold $M$, and it is nonempty
because it contains the boundary collar.  Hence $\Omega_0=M$, proving
\eqref{eq:V-nonzero-global}.

If $X(q)=0$ at some point, Lemma~\ref{lem:X-U-V} and
\eqref{eq:r-UV} would give $U(q)=V(q)=0$, contradicting the conclusion just
proved.  Hence $r=|X|$ is everywhere positive.  The definition of $e_0$ above is therefore valid and smooth on all of $M$. Set
$$
        b=\langle V,e_0\rangle,
        \qquad
        a=\langle U,e_0\rangle.
$$
Then
$r=(a^2+b^2)^{1/2}$ is smooth globally, and the
identities of Section~\ref{sec:rank-one} continue without further choices.
\end{proof}

The contact-angle normalization obtained from the Hopf boundary point lemma
after \eqref{eq:intro-ball} implies that the global flow starts transversely
from the boundary.

\begin{lemma}\label{lem:no-zero-angle}
In the non-free case,
\begin{equation}\label{eq:theta0-open}
        0<\theta_0<\pi,
\end{equation}
and therefore
\begin{equation}\label{eq:rprime-initial-negative}
        r'(0)=-\sin\theta_0<0.
\end{equation}
\end{lemma}

\begin{proof}
The normalization above gives \eqref{eq:theta0-open}.  Then
\eqref{eq:initial-a-b} and $r'=a/r$ give
\eqref{eq:rprime-initial-negative}.
\end{proof}

For later use, define
\begin{equation*}
        h_{\varepsilon,n}(r)=r^n\ee^{-\varepsilon r^2/2}.
\end{equation*}
For $0<r\le1$,
\begin{equation}\label{eq:h-n-increasing}
        h_{\varepsilon,n}'(r)=r^{n-1}\ee^{-\varepsilon r^2/2}(n-\varepsilon r^2)>0,
\end{equation}
so $h_{\varepsilon,n}$ is strictly increasing.

The first integral reduces the radial equation to one scalar first-order equation.
Since $E\ne0$ and $0<\theta_0<\pi$,
\begin{equation*}
        0<|E|<h_{\varepsilon,n}(1)=\ee^{-\varepsilon/2}.
\end{equation*}
By strict monotonicity \eqref{eq:h-n-increasing}, there is a unique
\begin{equation*}
        r_*\in(0,1)
        \qquad\text{such that}\qquad
        h_{\varepsilon,n}(r_*)=|E|.
\end{equation*}
The first integral gives
\begin{equation*}
        \sin\Theta=\frac{E}{h_{\varepsilon,n}(r)}
\end{equation*}
after fixing the continuous branch determined by the collar.  Hence
\begin{equation}\label{eq:r-first-order}
        (r')^2
        =1-\frac{E^2\ee^{\varepsilon r^2}}{r^{2n}}
        =1-\frac{E^2}{h_{\varepsilon,n}(r)^2}.
\end{equation}
In particular,
\begin{equation*}
        r\ge r_*>0.
\end{equation*}

\begin{lemma}\label{lem:phase-monotonicity}
Along every non-free orbit, $\sin\Theta$ has the constant sign of $E$, and
\begin{equation}\label{eq:theta-monotonicity-sign}
        \operatorname{sgn}(\Theta')=-\operatorname{sgn}(E).
\end{equation}
If $E>0$, then along the decreasing branch,
$$
        \frac\pi2<\Theta<\pi.
$$
If a radial critical point occurs, it is characterized by $\Theta=\pi/2$;
along any subsequent increasing branch,
$$
        0<\Theta<\frac\pi2.
$$
If $E<0$, the corresponding intervals are obtained by replacing $\Theta$ with
$-\Theta$.  In particular, an orbit has at most one radial critical point in
$0<r<1$.
\end{lemma}

\begin{proof}
The first integral and positivity of $h_{\varepsilon,n}(r)$ give
$$
        \sin\Theta=\frac{E}{h_{\varepsilon,n}(r)},
$$
so $\sin\Theta$ never vanishes and has the fixed sign of $E$.  Moreover,
$0<r\le1$, $n\ge2$, and $\varepsilon\in\{-1,0,1\}$ imply
$$
        \varepsilon r-\frac nr
        =\frac{\varepsilon r^2-n}{r}<0.
$$
The second equation in \eqref{eq:intrinsic-anciaux-system} therefore proves
\eqref{eq:theta-monotonicity-sign}.

Assume first that $E>0$.  At the initial boundary, $r'(0)<0$, so
$\cos\Theta(0)<0$, while $\sin\Theta(0)>0$.  After choosing the continuous
representative of the angle, this places $\Theta(0)$ in $(\pi/2,\pi)$.  Since
$\Theta$ is strictly decreasing, the first possible zero of $r'=\cos\Theta$
is reached at $\Theta=\pi/2$.  Once $\Theta$ has crossed $\pi/2$, one has
$\cos\Theta>0$, and strict monotonicity prevents a return to $\pi/2$.  Thus
there is at most one radial critical point.  The case $E<0$ follows by the
symmetry $(E,\Theta)\mapsto(-E,-\Theta)$ of the system.
\end{proof}

For the endpoint and inverse-flow arguments below, fix a smooth double
$\widehat M$ of $M$ (see \cite{Munkres}) and extend $e_0$ to a globally
defined smooth vector field $\widehat e_0$ on $\widehat M$.  Since
$\widehat M$ is compact, $\widehat e_0$ has a complete flow
$\widehat\varphi_s$.  Whenever such an orbit lies in $M$, it is independent
of the chosen extension, and we denote the restricted flow by $\varphi_s$.

\begin{proposition}\label{prop:radial-behavior}
Every integral curve of the global vector field $e_0$, issued from $\Sigma_0$ at time $s=0$, has the same radial function $r(s)$.  This function is defined on a common interval $[0,L]$, satisfies
\begin{equation}\label{eq:r-global-behavior}
        r(0)=r(L)=1,
        \qquad
        r_*\le r(s)<1\quad(0<s<L),
\end{equation}
reaches its unique minimum $r_*$ once, and obeys
\begin{equation}
        r'(0)<0,
        \qquad
        r'(L)>0.
\end{equation}
The return time is
\begin{equation}\label{eq:L-integral}
        L
        =2\int_{r_*}^{1}
        \frac{dr}{\sqrt{1-E^2\ee^{\varepsilon r^2}/r^{2n}}}<\infty.
\end{equation}
\end{proposition}

\begin{proof}
We divide the argument into four steps.

\smallskip
\noindent\emph{Step 1.}
By Proposition~\ref{prop:eta-closed}, the functions $a$, $b$, and $r$ are
constant along every leaf of $\Dcal$.  On the initial leaf $\Sigma_0$ their
values are
$$
        r(0)=1,
        \qquad
        a(0)=-\sin\theta_0,
        \qquad
        b(0)=-\cos\theta_0.
$$
They are therefore independent of the starting point $p\in\Sigma_0$.  The
intrinsic system is a smooth autonomous ODE on the region $r>0$.  By the
uniqueness and smooth-dependence theorem
\cite[Chapter~1]{CoddingtonLevinson1955}, all $e_0$-orbits issued from
$\Sigma_0$ have the same functions $(r,\Theta,a,b)$ on their common interval
of existence.  It is therefore meaningful to write these functions without a
subscript $p$.

\smallskip
\noindent\emph{Step 2.}
Lemma~\ref{lem:no-zero-angle} gives $r'(0)<0$.  The energy identity
\eqref{eq:r-first-order} implies $r\ge r_*$ and shows that $r'$ can vanish only
when $h_{\varepsilon,n}(r)=|E|$, hence only at $r=r_*$.  By
Lemma~\ref{lem:phase-monotonicity}, such a critical point can occur at most
once.  At that point $|\sin\Theta|=1$, and
$$
\begin{aligned}
        r''
        &=-\sin\Theta\,\Theta'\\
        &=-\left(\varepsilon r-\frac nr\right)\sin^2\Theta\\
        &=\frac nr-\varepsilon r
        \\&>0.
\end{aligned}
$$
Thus the critical point is a nondegenerate strict minimum.  The solution
arrives at $r_*$ in finite time.  To see this directly, put
$$
        G(r)=1-\frac{E^2}{h_{\varepsilon,n}(r)^2}.
$$
Then $G(r_*)=0$ and
$$
        G'(r_*)
        =2\frac{E^2h_{\varepsilon,n}'(r_*)}{h_{\varepsilon,n}(r_*)^3}>0.
$$
Hence $G(r)=G'(r_*)(r-r_*)+O((r-r_*)^2)$, so
$G(r)^{-1/2}=O((r-r_*)^{-1/2})$, which is integrable near \(r_*\). Since \(r'=-\sqrt{G(r)}\) on the
decreasing branch, the time needed to reach \(r_*\) is
\[
T_*=\int_{r_*}^{1}\frac{dr}{\sqrt{G(r)}}<\infty.
\]

\smallskip
\noindent\emph{Step 3.}
After the minimum, \(r'>0\), and no second turning point is
possible. On the decreasing and increasing branches, respectively,
\[
r'=-\sqrt{G(r)}
\qquad\text{and}\qquad
r'=\sqrt{G(r)}.
\]
Hence the travel time from \(r=1\) to \(r=r_*\) and the return
time from \(r=r_*\) to \(r=1\) are both
\(T_*\). Consequently, the first return to \(r=1\) occurs at
\[
L=2T_*,
\]
and \(r'(L)=-r'(0)>0\).

\smallskip
\noindent\emph{Step 4.}
Fix $p\in\Sigma_0$.  Since $e_0=-\nu$ on $\Sigma_0$, the curve
$\widehat\varphi_s(p)$ enters $M$ for sufficiently small positive $s$.
Suppose that it first leaves $M$ at some time $\tau<L$.  Then
$\widehat\varphi_\tau(p)\in\partial M$.  As long as the orbit remains in
$M$, it agrees with the $e_0$-orbit and hence
$$
        |X(\widehat\varphi_s(p))|=r(s).
$$
Therefore
$$
        |X(\widehat\varphi_\tau(p))|=r(\tau)<1,
$$
contradicting $X(\partial M)\subset\Sph^{2n-1}$.  Thus the orbit remains in
$M$ for $0\le s<L$.  Since $M$ is closed in $\widehat M$, its value at time
$L$ also lies in $M$.  By continuity,
$$
        |X(\widehat\varphi_L(p))|=r(L)=1.
$$
Since $|X|<1$ on $M^\circ$, this equality implies
$\widehat\varphi_L(p)\in\partial M$.  We may therefore define
$\varphi_s(p)=\widehat\varphi_s(p)$ for every $0\le s\le L$.
This proves the common interval, the endpoint signs, the uniqueness of the
minimum, and all assertions in \eqref{eq:r-global-behavior}--\eqref{eq:L-integral}.
\end{proof}

The estimates above are uniform for $\varepsilon\in\{-1,0,1\}$.  Indeed,
$n-\varepsilon r^2\ge n-1>0$ on $(0,1]$, so the turning point,
nondegeneracy, and finite return time require no sign-specific argument.

\subsection{Flow-out and the boundary-component theorem}

The remaining task for Theorem~\ref{thm:boundary-components} is global and topological.  At this stage we do not yet reconstruct the immersion as an Anciaux product and we do not use the relation between the two contact angles.  We only prove that the flow of the intrinsic rank-one direction carries the chosen boundary component through the whole manifold and reaches one connected terminal boundary component.  This suffices to complete the boundary count.

Proposition~\ref{prop:radial-behavior} shows that every orbit issued from
$\Sigma_0$ remains in $M$ for
$0\le s\le L$, and therefore permits the definition
\begin{equation}\label{eq:global-flow-map}
        \Phi:[0,L]\times\Sigma_0\longrightarrow M,
        \qquad
        \Phi(s,p)=\varphi_s(p).
\end{equation}

\begin{lemma}\label{lem:flow-preserves-D}
For every $s$,
\begin{equation}\label{eq:flow-preserve-eta}
        \varphi_s^*\eta=\eta
\end{equation}
whenever both sides are defined.  Consequently
$$
        d\varphi_s(\Dcal_p)=\Dcal_{\varphi_s(p)}.
$$
\end{lemma}

\begin{proof}
Cartan's formula and Proposition~\ref{prop:eta-closed} give
\begin{align*}
        \Lcal_{e_0}\eta
        &=\iota_{e_0}d\eta+d(\eta(e_0))
        \\&=0+d(1)
        \\&=0.
\end{align*}
Therefore $\eta$ is invariant under the flow, which proves \eqref{eq:flow-preserve-eta}.  Taking kernels gives invariance of $\Dcal$.
\end{proof}

\begin{lemma}
The map
$$
        F_L:\Sigma_0\longrightarrow\partial M,
        \qquad
        F_L(p)=\varphi_L(p),
$$
is an immersion whose image $\Sigma_1$ is a connected component of
$\partial M$.  Along $\Sigma_1$ the vector field $e_0$ is the outward unit
conormal.
\end{lemma}

\begin{proof}
Proposition~\ref{prop:radial-behavior} shows that $\varphi_L(p)$ is defined and
lies on $\partial M$ for every $p\in\Sigma_0$.  By
Lemma~\ref{lem:flow-preserves-D},
$$
        d\varphi_L(T_p\Sigma_0)=\Dcal_{\varphi_L(p)}.
$$
At the terminal point,
$$
        \nabla\rho=rr'e_0=r'(L)e_0,
$$
with $r'(L)>0$.  Since $\rho<1/2$ in the interior and $\rho=1/2$ on the
boundary, $\nabla\rho$ is a positive multiple of the outward conormal.
Therefore $e_0=\nu$ and
$$
        \Dcal=e_0^\perp=T\partial M
$$
along the image.  It follows that $F_L$ is a local diffeomorphism from
$\Sigma_0$ onto an open subset of $\partial M$.  Its image is compact and
therefore closed in $\partial M$.  Since $\Sigma_0$ is connected, the image is
connected; hence it is exactly one connected component, denoted $\Sigma_1$.
\end{proof}

\begin{proposition}\label{prop:global-diffeomorphism}
The map $\Phi$ in \eqref{eq:global-flow-map} is a diffeomorphism.  In particular,
\begin{equation}\label{eq:two-boundaries}
        \partial M=\Sigma_0\sqcup\Sigma_1,
        \qquad
        \Sigma_1=\Phi(\{L\}\times\Sigma_0).
\end{equation}
\end{proposition}

\begin{proof}
We first show that $\Phi$ is a local diffeomorphism.  Its differential sends $\partial_s$ to $e_0$.  By Lemma~\ref{lem:flow-preserves-D}, it maps $T_p\Sigma_0=\Dcal_p$ isomorphically onto $\Dcal_{\Phi(s,p)}$.  Since
$$
        T_{\Phi(s,p)}M
        =\R e_0\oplus\Dcal_{\Phi(s,p)},
$$
the differential is an isomorphism.  At $s=0$, $e_0=-\nu$ is the inward conormal of $M$, and the inward
conormal of the product domain at $\{0\}\times\Sigma_0$ is $\partial_s$.
At $s=L$, $r'(L)>0$ and
$$
        \nabla\rho=rr'e_0,
$$
so $e_0=\nu$ is the outward conormal of $M$, and the outward conormal of the
product domain at $\{L\}\times\Sigma_0$ is again $\partial_s$.  Hence
$d\Phi$ maps boundary tangent spaces to boundary tangent spaces and inward
half-spaces to inward half-spaces at both faces.  The inverse function theorem
for manifolds with boundary applies, so $\Phi$ is a local diffeomorphism on
the whole product.

We next prove injectivity.  Suppose
$$
        \Phi(s,p)=\Phi(t,q),
        \qquad
        0\le s\le t\le L.
$$
Applying the inverse flow for time $s$,
\begin{equation}\label{eq:injectivity-backflow}
        p=\varphi_{t-s}(q).
\end{equation}
If $0<t-s<L$, then Proposition~\ref{prop:radial-behavior} gives
$$
        |X(\varphi_{t-s}(q))|=r(t-s)<1,
$$
so the right-hand side of \eqref{eq:injectivity-backflow} lies in $M^\circ$, whereas $p\in\Sigma_0\subset\partial M$.  This is impossible.

If $t-s=L$, then necessarily $s=0$ and $t=L$.  Equality would identify a point of the initial boundary with a point of the terminal boundary.  At a point of $\Sigma_0$, the globally fixed vector field satisfies $e_0=-\nu$.  At a terminal point, the preceding transversality calculation gives $e_0=\nu$.  The same boundary point cannot satisfy both identities.  Therefore $t-s\ne L$.  We conclude that $s=t$, and uniqueness of the flow gives $p=q$.  Thus $\Phi$ is injective.

Because $[0,L]\times\Sigma_0$ is compact, the image of $\Phi$ is compact and hence closed in $M$.  Since $\Phi$ is a local diffeomorphism, its image is also relatively open in $M$, including near the boundary by the transversality just established.  The image is nonempty, and $M$ is connected, so the image is all of $M$.  Thus $\Phi$ is a bijective local diffeomorphism.  A bijective local diffeomorphism between compact manifolds with boundary is a diffeomorphism, proving the proposition and \eqref{eq:two-boundaries}.
\end{proof}

\begin{proof}[Proof of Theorem~\ref{thm:boundary-components}]
Fix a boundary component $\Sigma_0$.  If its contact angle is $\pi/2$,
\cite[Theorem~1.3]{GaoMaYao2026} implies that $X$ is a
diffeomorphism onto an equatorial Lagrangian $n$-disk, and hence $\partial M$
is connected.  If its contact angle is not $\pi/2$, Propositions~\ref{prop:global-nondegeneracy},~\ref{prop:radial-behavior}, and~\ref{prop:global-diffeomorphism} give a diffeomorphism $[0,L]\times\Sigma_0\longrightarrow M.$
Hence $\partial M$ has exactly the two components $\{0\}\times\Sigma_0$ and $\{L\}\times\Sigma_0$.  Thus in all cases
\(\partial M\) has at most two connected components,
which proves Theorem~\ref{thm:boundary-components}.
\end{proof}

\subsection{Proof of the global classification}

Theorem~\ref{thm:boundary-components} has now been proved.  We therefore start the classification from the established dichotomy that the boundary has one or two connected components.  The one-boundary case has been settled in \cite{GaoMaYao2026}.

\begin{proposition}[{\cite[Corollary~3.8]{GaoMaYao2026}}]
\label{prop:one-boundary-classification}
If $\partial M$ is connected, then its contact angle is $\pi/2$ and $X$ is a diffeomorphism onto an equatorial Lagrangian $n$-disk.
\end{proposition}

It remains to reconstruct the immersion in the two-boundary case and identify the contact angles. Define
\begin{equation*}
        \phi(s)=\int_0^s\frac{b(\tau)}{r(\tau)^2}\,d\tau,
        \qquad
        \gamma(s)=r(s)\ee^{i\phi(s)}.
\end{equation*}
Put
\begin{equation*}
        \psi(p)=X(p),
        \qquad p\in\Sigma_0.
\end{equation*}
Since $|X|=1$ on $\Sigma_0$, $\psi$ maps into $\Sph^{2n-1}$, and Theorem~\ref{thm:boundary-minimal} shows that it is minimal Legendrian.

\begin{proposition}\label{prop:global-product}
For all $(s,p)\in[0,L]\times\Sigma_0$,
\begin{equation}\label{eq:global-product-final}
        X\circ\Phi(s,p)=\gamma(s)\psi(p).
\end{equation}
The curve $\gamma$ is arclength-parametrized and satisfies the self-similar system \eqref{eq:intro-anciaux-system}.  Moreover,
\begin{equation}\label{eq:global-warped-metric}
        \Phi^*g=ds^2+r(s)^2g_\psi,
\end{equation}
where $g_\psi$ is the metric induced by the initial link.
\end{proposition}

\begin{proof}
Fix $p\in\Sigma_0$ and consider the vector-valued function
$$
        Q_p(s)=\ee^{-i\phi(s)}\frac{X(\Phi(s,p))}{r(s)}.
$$
Proposition~\ref{prop:local-reconstruction} was derived intrinsically and is
valid on every flow box.  It gives $$D_sQ_p=0$$ on each such box.  The interval
$[0,L]$ is compact, so finitely many overlapping flow boxes cover the orbit;
the local constants agree on the overlaps.  Thus $Q_p$ is constant on the
whole interval.  

Since $r(0)=1$ and $\phi(0)=0$,
$$
        Q_p(s)=Q_p(0)=X(p)=\psi(p).
$$
This proves
$$
        X(\Phi(s,p))=r(s)\ee^{i\phi(s)}\psi(p)=\gamma(s)\psi(p).
$$

The arclength assertion follows from
$$
\begin{aligned}
        \gamma'
        =\ee^{i\phi}(r'+ir\phi')
 =\ee^{i\phi}\left(\frac ar+i\frac br\right),
\end{aligned}
$$
and $a^2+b^2=r^2$.  Proposition~\ref{prop:intrinsic-anciaux} supplies the ODE
and the first integral.

It remains to record the metric.  Differentiating the global product gives
$$
        (X\circ\Phi)_*\partial_s=\gamma'\psi,
        \qquad
        (X\circ\Phi)_*W=\gamma\psi_*W
        \quad(W\in T\Sigma_0).
$$
The spherical relation $\langle\psi,\psi_*W\rangle=0$ and the Legendrian
relation $\langle J\psi,\psi_*W\rangle=0$ imply that the two displayed tangent
directions are orthogonal.  Their squared lengths are respectively $1$ and
$r^2|\psi_*W|^2$.  This proves \eqref{eq:global-warped-metric}.
\end{proof}

Since $\Sigma_0$ is connected, \eqref{eq:global-product-final} and the strict
unit-ball condition verify all the hypotheses of
Proposition~\ref{prop:product-boundary-angles}, with $s_0=0$ and $s_1=L$.
Hence the two boundary contact angles satisfy
$$
        \theta_0+\theta_1=\pi.
$$

\begin{proof}[Proof of Theorem~\ref{thm:classification}]
The boundary-component theorem has already been proved, so $\partial M$ has one or two connected components.

If $\partial M$ is connected, Proposition~\ref{prop:one-boundary-classification} gives alternative \rm(i).

Assume that $\partial M$ has two connected components.  No component can have
contact angle $\pi/2$, because
\cite[Theorem~1.3]{GaoMaYao2026}
would then imply that $X$ is a diffeomorphism onto an equatorial Lagrangian
disk, whose boundary is connected.  The non-free conclusions of Section~\ref{sec:global-classification}
therefore apply.  Proposition~\ref{prop:global-diffeomorphism} gives a diffeomorphism
$$
        \Phi:[0,L]\times\Sigma_0\longrightarrow M,
$$
with the two boundary faces mapped onto the two connected components of $\partial M$.  Set $N=\Sigma_0$.  Theorem~\ref{thm:boundary-minimal} shows that
$$
        \psi=X|_N:N\longrightarrow\Sph^{2n-1}
$$
is a compact minimal Legendrian immersion.  Proposition~\ref{prop:global-product} then yields
$$
        X\circ\Phi(s,p)=\gamma(s)\psi(p),
$$
where $\gamma$ is arclength-parametrized and satisfies the self-similar system \eqref{eq:intro-anciaux-system}.  Proposition~\ref{prop:radial-behavior} gives
$$
        r(0)=r(L)=1,
        \qquad
        0<r(s)<1\quad(0<s<L).
$$
Finally, Proposition~\ref{prop:product-boundary-angles}, applied with $s_0=0$ and $s_1=L$, gives
$$
        \cos\theta_0+\cos\theta_1=0,
        \qquad
        \theta_0+\theta_1=\pi.
$$
Thus alternative \rm(ii) holds.  The two alternatives are mutually exclusive because their boundary component counts are different.  This completes the proof.
\end{proof}

\subsection{Consequences of the classification}

\begin{corollary}\label{cor:liouville-dichotomy}
Under the hypotheses of Theorem~\ref{thm:classification}, the Liouville
form $\lambda=X^*\lambda_{\C^n}$ is exact.  Its relative class $\mathcal L_{\mathrm{rel}}(X)$ vanishes exactly in the equatorial Lagrangian disk case, equivalently, when $\partial M$ is connected. 
In the two-boundary case,
\(
        H^1_{\mathrm{dR}}(M,\partial M;\mathbb R)\cong\mathbb R
\)
and $\mathcal L_{\mathrm{rel}}(X)$ spans this space.
\end{corollary}

\begin{proof}
In the disk case, the image of $X$ lies in a Lagrangian linear subspace
through the origin, and hence $\lambda\equiv0$.

Assume the two-boundary alternative.  The global product formula gives
\[
        \Phi^*\lambda=r(s)\sin\Theta(s)\,ds.
\]
Define
\[
        G(s)=\int_0^s r(\tau)\sin\Theta(\tau)\,d\tau
\]
and define $u$ on $M$ by $u\circ\Phi(s,p)=G(s)$.  Then $du=\lambda$, so
$\lambda$ is exact.

Since the first integral $E$ is nonzero,
\(
        \sin\Theta
        =E r^{-n}\ee^{\varepsilon r^2/2}
\)
has a fixed nonzero sign.  Therefore $G(L)\ne G(0)$.  Every primitive of
$\lambda$ differs from $u$ by a constant because $M$ is connected, so no
primitive takes the same value on the two boundary components.  Thus
$\mathcal L_{\mathrm{rel}}(X)\ne0$.

Finally,
\[
 (M,\partial M)\cong
 ([0,L]\times N,\{0,L\}\times N),
\]
and $N$ is connected.  The relative K\"unneth formula therefore gives
\(
 H^1_{\mathrm{dR}}(M,\partial M;\mathbb R)\cong\mathbb R.
\)
\end{proof}

\begin{corollary}
\label{cor:boundary-congruence}
In the two-boundary case, define
\[
        j_0(p)=\Phi(0,p),
        \qquad
        j_1(p)=\Phi(L,p).
\]
Then there exists $\vartheta\in\mathbb R$ such that
\(
        X\circ j_1
        =\ee^{i\vartheta}X\circ j_0
\).
Consequently, after the identification induced by the global flow, the two
boundary immersions are unitarily congruent.  In particular, they induce the
same metric and have the same $(n-1)$-dimensional volume.
\end{corollary}

\begin{proof}
Since $r(0)=r(L)=1$, the global product formula gives
\[
        X(\Phi(0,p))=\ee^{i\phi(0)}\psi(p),
        \qquad
        X(\Phi(L,p))=\ee^{i\phi(L)}\psi(p).
\]
Thus the conclusion holds with
$\vartheta=\phi(L)-\phi(0)$.
\end{proof}

\section{Low-dimensional cases and Calabi-suspended examples}\label{sec:calabi-realizations}
  We begin with the two-dimensional classification and the angle estimate before turning to the construction of the Calabi-suspended families.

\subsection{The two-dimensional case}
\begin{proof}[Proof of Corollary~\ref{cor:two-dimensional}]
When $n=2$, the link $N$ is a connected closed one-dimensional manifold, hence is diffeomorphic to $S^1$.  By Theorem~\ref{thm:boundary-minimal},
$$
        \psi:S^1\longrightarrow\Sph^3
$$
is a one-dimensional minimal immersion.  A one-dimensional minimal submanifold of a Riemannian manifold is a geodesic, so the image is a great circle.  The immersion of the compact circle onto that great circle is a finite covering.  Let $m\ge1$ be its degree.  After a unitary transformation, and after parametrizing the domain by $t\in\R/(2\pi\mathbb Z)$,
$$
        \psi(t)=(\cos mt,\sin mt).
$$
Substituting this expression into \eqref{eq:global-product-final} gives
\eqref{eq:intro-two-dimensional}.  When $\varepsilon=0$, equations
\eqref{eq:alpha-prime-product} and \eqref{eq:polar-basic}, with $n=2$, give
$$
        \alpha'=-\frac{\sin\Theta}{r},
        \qquad
        \phi'=\frac{\sin\Theta}{r}.
$$
Thus $\alpha+\phi$ is constant.  Multiplying the immersion by a constant
unitary scalar changes this constant by twice the scalar phase, so we may
normalize $\alpha+\phi=0$.  Since $\Theta=\alpha-\phi$, it follows that
$\Theta=-2\phi$.  The first integral then gives
$$
        E=r^2\sin\Theta=-r^2\sin(2\phi)
        =-\operatorname{Im}(\gamma^2).
$$
Writing $\gamma=x+iy$, this is the hyperbola $2xy=-E$.  The product of this
profile with the unit great circle is the Lagrangian catenoid of
Harvey--Lawson and Castro--Urbano
\cite{HarveyLawson1982,CastroUrbano1999}: restricting between the two unit
sphere intersections gives the capillary annulus in Luo--Sun
\cite[Example~3.1]{LuoSun2021}.  When $\varepsilon\ne0$, the curve equation is
precisely the rotational self-similar system in Anciaux's construction
\cite{Anciaux2006}.  Passing from the great circle to its degree-$m$
parametrization produces the stated finite covers.
\end{proof}
\subsection{Angle estimates}
\begin{corollary}\label{cor:angle-formula}
Assume alternative \rm(ii) of Theorem~\ref{thm:classification}, and let
$r_*:=\min_{[0,L]}r$.  Then
\begin{equation}\label{eq:intro-angle-formula}
        |\cos\theta_0|=|\cos\theta_1|
        =r_*^n\exp\!\left(\frac{\varepsilon(1-r_*^2)}2\right).
\end{equation}
If $\dist(X(M),0)\ge\rho_0>0$, then
\begin{equation}\label{eq:intro-angle-gap}
        |\cos\theta_i|
        \ge
        \rho_0^n\exp\!\left(\frac{\varepsilon(1-\rho_0^2)}2\right),
        \qquad i=0,1.
\end{equation}
\end{corollary}

\begin{proof}[Proof of Corollary~\ref{cor:angle-formula}]
At the unique minimum point of $r$,
$$
        r'=\cos\Theta=0,
        \qquad
        |\sin\Theta|=1.
$$
The first integral therefore gives
\begin{equation}\label{eq:E-rstar-final}
        |E|=r_*^n\ee^{-\varepsilon r_*^2/2}.
\end{equation}
At the endpoints, \eqref{eq:product-angle-components} and \eqref{eq:anciaux-first-integral} give
$$
        |\cos\theta_i|=\ee^{\varepsilon/2}|E|,
        \qquad i=0,1.
$$
Combining this with \eqref{eq:E-rstar-final} proves \eqref{eq:intro-angle-formula}.

Define
$$
        g_{\varepsilon,n}(r)=r^n\exp\!\left(\frac{\varepsilon(1-r^2)}{2}\right).
$$
For $0<r\le1$,
$$
        g_{\varepsilon,n}'(r)
        =r^{n-1}\exp\!\left(\frac{\varepsilon(1-r^2)}{2}\right)(n-\varepsilon r^2)>0.
$$
If $\dist(X(M),0)\ge\rho_0$, then $r_*\ge\rho_0$.  Monotonicity of $g_{\varepsilon,n}$ gives \eqref{eq:intro-angle-gap}.
\end{proof}
\subsection{Calabi suspension and higher dimensional examples}
We now prove Theorem~\ref{thm:intro-calabi-realization}.  The construction
below is the link-level Calabi product with a point (compare Li--Wang
\cite{LiWang2011Calabi}).  We include the proof because the precise
normalization and the induced metric will be used in the iteration.

\begin{proposition}\label{prop:calabi-suspension}
Let $k\ge2$, and let
$$
        \psi:N^{k-1}\longrightarrow\Sph^{2k-1}\subset\C^k
$$
be a compact minimal Legendrian immersion.  Then the map
\begin{equation}
        \Psi(t,p)
        =\mathcal C_k(\psi)(t,p)
        =
        \left(
        \sqrt{\frac{k}{k+1}}\,\ee^{it}\psi(p),
        \frac1{\sqrt{k+1}}\ee^{-ikt}
        \right)
\end{equation}
defines a compact minimal Legendrian immersion
$$
        \Psi:S^1\times N\longrightarrow\Sph^{2k+1}\subset\C^{k+1}.
$$
Its induced metric is
\begin{equation}\label{eq:calabi-metric}
        g_\Psi=k\,dt^2+\frac{k}{k+1}g_\psi.
\end{equation}
\end{proposition}

\begin{proof}
Set
$$
        a=\sqrt{\frac{k}{k+1}},
        \qquad
        b=\frac1{\sqrt{k+1}}.
$$
Then $a^2+b^2=1$, so $|\Psi|=1$.  For $v\in T_pN$,
\begin{equation}\label{eq:calabi-differentials}
        \Psi_*v=(a\ee^{it}\psi_*v,0),
        \qquad
        \Psi_t=(ia\ee^{it}\psi,-ikb\ee^{-ikt}).
\end{equation}
Since $\psi$ is Legendrian,
$$
        \langle J\psi,\psi_*v\rangle=0,
        \qquad
        \omega(\psi_*v,\psi_*w)=0.
$$
Consequently,
$$
        \langle J\Psi,\Psi_*v\rangle=0,
        \qquad
        \omega(\Psi_*v,\Psi_*w)=0.
$$
Moreover,
$$
\begin{aligned}
        \langle J\Psi,\Psi_t\rangle
        &=a^2-kb^2\\
        &=\frac{k}{k+1}-\frac{k}{k+1}
        \\&=0.
\end{aligned}
$$
Thus the standard contact form of $\Sph^{2k+1}$ vanishes on all tangent vectors to $\Psi$, and $\Psi$ is Legendrian.

The formulas in \eqref{eq:calabi-differentials} also give
$$
        |\Psi_t|^2=a^2+k^2b^2=k,
        \qquad
        \langle\Psi_t,\Psi_*v\rangle=0,
$$
and
$$
        \langle\Psi_*v,\Psi_*w\rangle
        =a^2\langle\psi_*v,\psi_*w\rangle
        =\frac{k}{k+1}\langle\psi_*v,\psi_*w\rangle.
$$
This proves \eqref{eq:calabi-metric}, and in particular proves that $\Psi$ is an immersion.

It remains to verify minimality.  The assertion is local, so choose a local oriented orthonormal frame
$$
        v_1,\ldots,v_{k-1}
$$
for the metric induced by $\psi$.  Since $\psi$ is minimal Legendrian, the cone over $\psi$ has locally constant Lagrangian phase.  Thus, for a locally constant function $\beta$,
\begin{equation}\label{eq:old-link-phase}
        \Omega_k(\psi,\psi_*v_1,\ldots,\psi_*v_{k-1})
        =\ee^{i\beta},
\end{equation}
where $\Omega_k=dz_1\wedge\cdots\wedge dz_k.$

At radius one, an orthonormal frame for the cone over $\Psi$ is
$$
        \Psi,
        \qquad
        E_t=\frac1{\sqrt{k}}\Psi_t,
        \qquad
        E_j=(\ee^{it}\psi_*v_j,0),
        \quad 1\le j\le k-1.
$$
Let $\Omega_{k+1}=dz_1\wedge\cdots\wedge dz_{k+1}.$ We expand the determinant along the last complex coordinate.  The contribution of the last coordinate of $\Psi$ is
$$
        (-1)^k\frac{iab}{\sqrt{k}}
        \Omega_k(\psi,\psi_*v_1,\ldots,\psi_*v_{k-1}),
$$
while the contribution of the last coordinate of $E_t$ is
$$
        (-1)^k i\sqrt{k}\,ab
        \Omega_k(\psi,\psi_*v_1,\ldots,\psi_*v_{k-1}).
$$
Here the factors $\ee^{-ikt}$ from the last coordinate and $\ee^{ikt}$ from the first $k$ coordinates cancel in both terms.  Consequently,
\begin{equation*}
\begin{aligned}
\Omega_{k+1}(\Psi,E_t,E_1,\ldots,E_{k-1})
=(-1)^k i\frac{(k+1)ab}{\sqrt{k}}
\Omega_k(\psi,\psi_*v_1,\ldots,\psi_*v_{k-1})
=(-1)^k i\ee^{i\beta},
\end{aligned}
\end{equation*}
where we used $(k+1)ab/\sqrt{k}=1$ and \eqref{eq:old-link-phase}.  Thus the cone over $\Psi$ has locally constant Lagrangian phase.  Hence the cone is minimal Lagrangian, and $\Psi$ is minimal Legendrian.  Compactness follows from compactness of $S^1\times N$.
\end{proof}

\begin{corollary}\label{cor:iterated-calabi-suspension}
Let $k_0\ge2$, and let
$$
        \psi_{k_0}:N^{k_0-1}\longrightarrow\Sph^{2k_0-1}
$$
be a compact minimal Legendrian immersion.  Define recursively
$$
        N^{j}=S^1\times N^{j-1},
        \qquad
        \psi_{j+1}=\mathcal C_j(\psi_j),
        \qquad j\ge k_0.
$$
Then, for every $n\ge k_0$,
$$
        N^{n-1}\cong\mathbb T^{n-k_0}\times N^{k_0-1}
$$
and
$$
        \psi_n:N^{n-1}\longrightarrow\Sph^{2n-1}
$$
is a compact minimal Legendrian immersion.
\end{corollary}

\begin{proof}
Proposition~\ref{prop:calabi-suspension} raises both the link dimension and the complex ambient dimension by one.  Iterating it $n-k_0$ times proves the assertion.  Each step adds one circle factor to the domain, hence
$
        N^{n-1}\cong(S^1)^{n-k_0}\times N^{k_0-1}
        =\mathbb T^{n-k_0}\times N^{k_0-1}.
$
\end{proof}

\begin{lemma}\label{lem:existence-anciaux-segments}
Let $n\ge2$ and $\varepsilon\in\{-1,0,1\}$.  For every
$$
        0<E<\ee^{-\varepsilon/2},
$$
there is an arclength-parametrized self-similar curve
$$
        \gamma_{n,\varepsilon,E}:[0,L_{n,\varepsilon,E}]\longrightarrow\C^*
$$
satisfying
$$
        r'=\cos\Theta,
        \qquad
        \Theta'=\left(\varepsilon r-\frac nr\right)\sin\Theta,
$$
and
$$
        r(0)=r(L_{n,\varepsilon,E})=1,
        \qquad
        0<r(s)<1\quad(0<s<L_{n,\varepsilon,E}).
$$
The minimum radius is the unique $r_*\in(0,1)$ determined by
\begin{equation}\label{eq:anciaux-segment-minimum}
        r_*^n\ee^{-\varepsilon r_*^2/2}=E,
\end{equation}
and
\begin{equation}\label{eq:anciaux-segment-length}
        L_{n,\varepsilon,E}
        =2\int_{r_*}^{1}
        \frac{dr}{\sqrt{1-E^2\ee^{\varepsilon r^2}/r^{2n}}}<\infty.
\end{equation}
\end{lemma}

\begin{proof}

The function
$$
        h_{\varepsilon,n}(r)=r^n\ee^{-\varepsilon r^2/2}
$$
is strictly increasing on $(0,1]$, because
$$
        h_{\varepsilon,n}'(r)
        =r^{n-1}\ee^{-\varepsilon r^2/2}(n-\varepsilon r^2)>0
$$
for $n\ge2$ and $\varepsilon\in\{-1,0,1\}$.
Hence there is a unique $r_*\in(0,1)$ satisfying \eqref{eq:anciaux-segment-minimum}.

By the classical ODE existence and uniqueness theorem
\cite[Chapter~1]{CoddingtonLevinson1955}, the self-similar system has a unique local solution with
initial data
$$
        r_+(0)=r_*,
        \qquad
        \Theta_+(0)=\frac\pi2.
$$
Its first integral is
$$
        r_+^n\ee^{-\varepsilon r_+^2/2}\sin\Theta_+=E.
$$
Furthermore,
$$
        r_+''(0)
        =-\sin\Theta_+(0)\Theta_+'(0)
        =\frac n{r_*}-\varepsilon r_*>0,
$$
so $r_*$ is a strict local minimum.  Since $E>0$,
$\sin\Theta_+>0$, while
$\varepsilon r-n/r<0$ on $0<r\le1$; hence $\Theta_+$ is strictly decreasing.
Immediately after the turning point one has $0<\Theta_+<\pi/2$, so
$r_+'>0$ and the solution follows the increasing branch
\begin{equation}
        (r_+')^2
        =1-\frac{E^2\ee^{\varepsilon r_+^2}}{r_+^{2n}}.
\end{equation}
The right-hand side is positive on $(r_*,1]$ and has a simple zero at $r_*$.
Thus $r_+$ cannot turn before reaching $1$, and separation of variables shows
that it reaches $r=1$ in the finite time
$$
        T_{n,\varepsilon,E}=\int_{r_*}^{1}
        \frac{dr}{\sqrt{1-E^2\ee^{\varepsilon r^2}/r^{2n}}}.
$$

Define a solution on $[0,2T_{n,\varepsilon,E}]$ by
$$
(r,\Theta)(s)=
\begin{cases}
\bigl(r_+(T_{n,\varepsilon,E}-s),\,\pi-\Theta_+(T_{n,\varepsilon,E}-s)\bigr),&0\le s\le T_{n,\varepsilon,E},\\
\bigl(r_+(s-T_{n,\varepsilon,E}),\,\Theta_+(s-T_{n,\varepsilon,E})\bigr),&T_{n,\varepsilon,E}\le s\le2T_{n,\varepsilon,E}.
\end{cases}
$$
The two definitions agree smoothly at $s=T_{n,\varepsilon,E}$, and direct substitution verifies the self-similar system on both halves.  Thus $L_{n,\varepsilon,E}=2T_{n,\varepsilon,E}$, the endpoints have radius one, and the radius is strictly smaller than one in the interior.

Finally, define $\phi$ by
$$
        \phi'=\frac{\sin\Theta}{r},
        \qquad
        \phi(0)=0,
$$
and put $\gamma_{n,\varepsilon,E}=r\ee^{i\phi}$.  Then
$$
        \gamma_{n,\varepsilon,E}'
        =\ee^{i\phi}(r'+ir\phi')
        =\ee^{i\phi}(\cos\Theta+i\sin\Theta),
$$
so $|\gamma_{n,\varepsilon,E}'|=1$.  This gives the required arclength-parametrized curve and proves \eqref{eq:anciaux-segment-length}.
\end{proof}
\begin{corollary}\label{cor:general-link-realization}
Under the assumptions of Corollary~\ref{cor:iterated-calabi-suspension}, assume in addition that $N^{k_0-1}$ is connected.  Let $n\ge k_0$, $\varepsilon\in\{-1,0,1\}$, and $0<E<\ee^{-\varepsilon/2}$.  Let $\gamma_{n,\varepsilon,E}$ be the unit-ball self-similar segment supplied by Lemma~\ref{lem:existence-anciaux-segments}.  Then
$$
        X_{n,\varepsilon,E}(s,p)=\gamma_{n,\varepsilon,E}(s)\psi_n(p)
$$
defines a compact immersed Lagrangian solution of
$H+\varepsilon X^\perp=0$,
$$
        X_{n,\varepsilon,E}:
        [0,L_{n,\varepsilon,E}]\times\mathbb T^{n-k_0}\times N^{k_0-1}
        \longrightarrow\overline{\B}^{2n}.
$$
It has exactly two boundary components, both diffeomorphic to $\mathbb T^{n-k_0}\times N^{k_0-1}$, along which  $X_{n,\varepsilon,E}$ satisfies the Legendrian capillary boundary condition with supplementary constant contact angles.
\end{corollary}

\begin{proof}
The iterated map $\psi_n$ is compact minimal Legendrian by Corollary~\ref{cor:iterated-calabi-suspension}.  Since $N^{k_0-1}$ is connected, so is
$$
        N^{n-1}\cong\mathbb T^{n-k_0}\times N^{k_0-1}.
$$
Hence the product parameter manifold is connected and has exactly the two boundary components corresponding to the endpoints of the profile interval.  Proposition~\ref{prop:anciaux-product} proves that $X_{n,\varepsilon,E}$ satisfies $H+\varepsilon X^\perp=0$.  The radial properties of the profile place its image in $\overline{\B}^{2n}$ and its interior in $\B^{2n}$.  Proposition~\ref{prop:product-boundary-angles} gives the constant-angle Legendrian capillary boundary condition, and \eqref{eq:product-supplementary} gives supplementary contact angles.
\end{proof}

\begin{proof}[Proof of Theorem~\ref{thm:intro-calabi-realization}]
Proposition~\ref{prop:calabi-suspension} and
Corollary~\ref{cor:iterated-calabi-suspension} produce the asserted compact
minimal Legendrian immersion $\psi_n$ for every $n\ge k_0$.  Fix such an $n$
and $\varepsilon\in\{-1,0,1\}$.  For every
$0<E<\ee^{-\varepsilon/2}$,
Lemma~\ref{lem:existence-anciaux-segments} supplies an
arclength-parametrized unit-ball profile, and
Corollary~\ref{cor:general-link-realization} pairs it with $\psi_n$ to produce
a compact immersed solution with two supplementary Legendrian capillary
boundary components.  As $E$ varies, these profiles give the claimed
one-parameter family.
\end{proof}

\noindent{\bf Acknowledgments}
This work was supported by the National Natural Science Foundation of China (Grant Nos. 12271069, 12201138, 12401057, 12471048, W2521103), the Natural Science Foundation of Henan Province (Grant No. 262300421869) and the Beijing Natural Science Foundation (Grant No. 1244039).



\begingroup
\small
\raggedright
\authorinfo
  {Dong Gao}
  {School of Science, Beijing University of Civil Engineering and Architecture,
   Beijing 102616, P.R. China}
  {gaodong@bucea.edu.cn}
\authorinfo
  {Yong Luo}
  {Mathematical Science Research Center of Mathematics,
   Chongqing University of Technology, Chongqing 400054, P.R. China}
  {yongluo-math@cqut.edu.cn}
\authorinfo
  {Hui Ma}
  {Department of Mathematical Sciences, Tsinghua University,
   Beijing 100084, P.R. China}
  {ma-h@tsinghua.edu.cn}
\authorinfo
  {Jiabin Yin}
  {School of Mathematics and Statistics, Xinyang Normal University,
   Xinyang 464000, P.R. China}
  {jiabinyin@126.com}
\endgroup

\end{document}